\documentclass[reqno,12pt]{amsart}

\usepackage{color} 
\usepackage{ifpdf}
\ifpdf
    \usepackage{graphicx}
    \usepackage[pdftex]{hyperref}
    \hypersetup{
        unicode=false,          
        pdftoolbar=true,        
        pdfmenubar=true,        
        pdffitwindow=false,     
        pdfstartview={FitH},    
        pdftitle={MCP Article},      
        pdfauthor={Michael Holst},   
        pdfsubject={Mathematics},    
        pdfcreator={Michael Holst},  
        pdfproducer={Michael Holst}, 
        pdfkeywords={PDE, analysis, mathematical physics}, 
        pdfnewwindow=true,      
        colorlinks=true,        
        linkcolor=red,          
        citecolor=blue,         
        filecolor=magenta,      
        urlcolor=cyan           
    }

\else
    \usepackage{graphicx}
    \usepackage{pstricks}
    
    \newcommand{\href}[2]{#2}
\fi

\usepackage{times}
\usepackage{amsfonts}

\usepackage{amsmath}
\usepackage{amsthm}
\usepackage{amssymb}
\usepackage{amsbsy}
\usepackage{amscd}

\usepackage{enumerate}
\usepackage{verbatim}

\newtheorem{theorem}{Theorem}[section]

\newtheorem{lemma}[theorem]{Lemma}

\newtheorem{assumption}[theorem]{Assumption}

\numberwithin{equation}{section}  

  \newcounter{mnote}
  \let\oldmarginpar\marginpar
    \renewcommand\marginpar[1]{\-\oldmarginpar[\raggedleft\footnotesize #1]%
    {\raggedright\footnotesize #1}}

\newenvironment{enumerateX}
{\begin{list}{\arabic{enumi})}{
\usecounter{enumi}
\leftmargin 2.5em\topsep 0.5em\itemsep -0.0em\labelwidth 50.0em}}
{\end{list}}

\newenvironment{itemizeX}
{\begin{list}{\labelitemi}{
\leftmargin 2.5em\topsep 0.5em\itemsep -0.0em\labelwidth 50.0em}}
{\end{list}}

\definecolor{myblue}{rgb}{0.2,0.2,0.7}
\definecolor{mygreen}{rgb}{0,0.6,0}
\definecolor{mycyan}{rgb}{0,0.6,0.6}
\definecolor{myred}{rgb}{0.9,0.2,0.2}
\definecolor{mymagenta}{rgb}{0.9,0.2,0.9}
\definecolor{mywhite}{rgb}{1.0,1.0,1.0}
\definecolor{myblack}{rgb}{0.0,0.0,0.0}

\newcommand{\beq}{\begin{equation}}
\newcommand{\eeq}{\end{equation}}
\newcommand{\beqa}{\begin{eqnarray}}
\newcommand{\eeqa}{\end{eqnarray}}
\renewcommand{\div}{{\operatorname{div}}}

\newcommand{\eps}{\varepsilon}

\newcommand{\T}{{\mathbb T}}

\newcommand{\PP}{{\mathbb P}}       
\newcommand{\R}{{\mathbb R}}       

\newcommand{\V}{{\mathbb V}}

\newcommand{\cA}{{\mathcal A}}

\newcommand{\cE}{{\mathcal E}}

\newcommand{\cI}{{\mathcal I}}

\newcommand{\cL}{{\mathcal L}}
\newcommand{\cM}{{\mathcal M}}

\newcommand{\cR}{{\mathcal R}}

\newcommand{\cT}{{\mathcal T}}

\newcommand{\half}{{\scriptstyle 1/2}}

\DeclareMathAlphabet{\mathpzc}{OT1}{pzc}{m}{it}

\newcommand{\bit}{\begin{itemize}}
\newcommand{\eit}{\end{itemize}}
\newcommand{\iic}{\emph}

\newcommand{\f}{\frac}

\newcommand{\oor}{\text{ or }}

\newcommand{\an}{\text{ and }}

\newcommand{\where}{\text{\ where }}

\newcommand{\st}{\text{\ such that }}

\newcommand{\tforall}{\text{ for all }}

\newcommand{\rest}{\big|}

\newcommand{\loc}{_{\text{loc}}}

\newcommand{\forma}{a(\, \cdot \, , \, \cdot \,)}
\newcommand{\dualforma}{a^\ast(\, \cdot \, , \, \cdot \,)}

\newcommand{\essinf}{\text{ess inf\,}}
\newcommand{\essup}{\text{ess sup\,}}
\newcommand{\osc}{\text{osc}}
\newcommand{\tosc}{\text{\em osc}}

\newcommand{\grad}{\nabla} 

\newcommand{\goto}{\rightarrow}

\newcommand{\lbb}{\llbracket}
\newcommand{\rbb}{\rrbracket}

\newcommand{\norm}[1]{\ensuremath{\lVert{#1} \rVert}}

\newcommand{\nr}[1]{\norm{#1}} 

\newcommand{\nrse}[1]{| \! | \! | {#1} | \! | \! |} 

\newcommand{\pa}{\partial}

\newcommand{\vareps}{ \varepsilon }

\newenvironment{qqq}{\begin{eqnarray*}\begin{split}\end{split}}{\end{eqnarray*}}
\newcommand\bqq{\begin{qqq}}
\newcommand\eqq{\end{qqq}}

\newenvironment{dsub}[2]{  \begin{array}{ccccccccccccccc}{#1} \\ {#2}}{\end{array} }
\newcommand\bml{\begin{dsub}}
\newcommand\eml{\end{dsub}}

\newenvironment{mat}{\left(\begin{array}{ccccccccccccccc}}{\end{array}\right)}
\newcommand\bcm{\begin{mat}}
\newcommand\ecm{\end{mat}}

\newenvironment{rmat}{\left(\begin{array}{rrrrrrrrrrrrr}}{\end{array}\right)}
\newcommand\brm{\begin{rmat}}
\newcommand\erm{\end{rmat}}

\definecolor{blue}{rgb}{0.2,0.2,0.7}
\definecolor{red}{rgb}{0.7,0.3,0.1}
\definecolor{cyan}{rgb}{0.2,0.5,0.6}

\usepackage{mathtools}
\usepackage{stmaryrd}

\begin{document}

\title[GOAFEM for Nonsymmetric Problems]
      {Convergence of Goal-Oriented Adaptive Finite Element Methods
       for Nonsymmetric Problems}

\author[M. Holst]{Michael Holst}
\email{mholst@math.ucsd.edu}

\author[S. Pollock]{Sara Pollock}
\email{snpolloc@math.ucsd.edu}

\address{Department of Mathematics\\
         University of California San Diego\\ 
         La Jolla CA 92093}

\thanks{MH was supported in part by NSF Awards~0715146 and 0915220,
and by DOD/DTRA Award HDTRA-09-1-0036.}
\thanks{SP was supported in part by NSF Award~0715146.}

\date{\today}


\begin{abstract}
In this article we develop convergence theory for a class of 
goal-oriented adaptive finite element algorithms for second 
order nonsymmetric linear elliptic equations.
In particular, we establish contraction results 
for a method of this type for Dirichlet problems
involving the elliptic operator 
$
\cL u = \grad \cdot (A \grad u) - b \cdot \grad u - cu,
$
with $A$ Lipschitz, almost-everywhere symmetric positive definite,
with $b$ divergence-free, and with $c \ge 0$.
We first describe the problem class and review some standard
facts concerning conforming finite element discretization and
error-estimate-driven adaptive finite element methods (AFEM).
We then describe a goal-oriented variation of standard AFEM (GOAFEM).
Following the recent work of Mommer and Stevenson for symmetric problems, 
we establish contraction of GOAFEM and convergence in the sense of the goal function.
Our analysis approach is signficantly different from that of Mommer and
Stevenson, combining the recent
contraction frameworks developed by Cascon, Kreuzer, Nochetto and Siebert;
by Nochetto, Siebert and Veeser;
and by Holst, Tsogtgerel and Zhu.
We include numerical results demonstrating performance of our method with standard goal-oriented strategies on a convection problem .
\end{abstract}

\maketitle

\vspace*{-1.2cm}
{\scriptsize
\tableofcontents
}

\section{Introduction}
\label{sec:intro1}

In this article we develop convergence theory for a class of 
goal-oriented adaptive finite element methods for 
second order nonsymmetric linear elliptic equations.
In particular, we report contraction results 
for a method of this type for the problem 
\begin{align}
-\grad \cdot (A \grad u) + b \cdot \grad u + cu &= f,
\quad \mbox{ in } \Omega, \label{eqn:strong-eqn1} \\
u &= 0, \quad \mbox{ on } \partial \Omega, \label{eqn:strong-b1c}
\end{align}
with $\Omega \subset \R^d$  a polyhedral domain, $d = 2 \oor 3$,
with $A$ Lipschitz, almost-everywhere symmetric positive definite (SPD),
with $b$ divergence-free, and with $c \ge 0$.
The standard weak formulation of this problem reads:  Find $u \in H^1_0(\Omega)$ such that
\begin{align}
a(u,v) &= f(v), \quad \forall v \in H_0^1(\Omega),
\label{primal_proble1m}
\end{align}
where
\begin{equation}
a(u,v) = \int_{\Omega} A \nabla u \cdot \nabla v
          + b \cdot \nabla u v + c u v ~dx,
\qquad
f(v) = \int_{\Omega} f v ~dx.
\label{primal_form1}
\end{equation}

Our approach is to first describe the problem class in some detail,
and  review some standard facts concerning conforming finite element 
discretization and error-estimate-driven adaptive finite element methods (AFEM).
We will then describe a goal-oriented variation of standard AFEM (GOAFEM).
Following the recent work of Mommer and Stevenson~\cite{MoSt09} for 
symmetric problems, we  establish contraction of GOAFEM and convergence in the sense of the goal function.
Our analysis approach is signficantly different from that of Mommer and
Stevenson~\cite{MoSt09}, combining the recent contraction frameworks 
of Cascon, Kreuzer, Nochetto and Siebert ~\cite{CKNS08},
of Nochetto, Siebert and Veeser~\cite{NSV09},
and of Holst, Tsogtgerel and Zhu~\cite{HTZ09a}.
We also give some numerical results comparing our goal-oriented method both to the one presented in~\cite{MoSt09} and the dual weighted residual (DWR) method as in~\cite{BaRa03,Becker.R;Rannacher.R1996a,EHL02,Giles.M;Suli.E2003,Gratsch.T;Bathe.K2005,EHM01}, among others.  Unlike the existing literature on the DWR method, we prove strong convergence of our goal-oriented method.   We establish contraction of the goal error in terms of the energy norm errors and error estimators of the primal and dual problems, and indicate how this implies optimality in terms of the global error.   Controlling this overestimate of the error shows convergence of the method to the goal, although not optimality in this sense.  Our numerical results demonstrate, however, that the algorithm presented here performs at least comparably to and in some cases better than DWR and the method in~\cite{MoSt09} on a variety of convection dominated linear problems.

The goal-oriented problem concerns achieving a target quality in a given linear functional $g \colon H^1_0(\Omega) \to \R$ of the weak solution $u \in H_0^1(\Omega)$ of the problem~\eqref{primal_proble1m}.
For example, $g(u) =\int_\Omega \f 1 {|\omega|}\chi_\omega u$, the average value of $u$ over some subdomain $\omega \subset \Omega$. By writing down the adjoint operator, $a^\ast(z,v) = a(v,z)$, we consider the \emph{adjoint} or \emph{dual} problem: find $z \in H_0^1(\Omega)$ such that $a^\ast(z,v) = g(v), \tforall v \in H_0^1(\Omega)$.
It has been shown for the symmetric form ($b=0$) of problem~\eqref{eqn:strong-eqn1}--\eqref{eqn:strong-b1c} with piecewise constant SPD diffusion cofficient $A$ (and with $c=0$), that by solving the \emph{primal} and \emph{dual} problems simultaneously, one may converge to an approximation of $g(u)$ faster than by approximating $u$ and then $g(u)$, when forcing contraction in only the primal problem~\cite{MoSt09}.
We will follow the same general approach to establish similar goal-oriented AFEM results for nonsymmetric problems.
In order to handle nonsymmetry, we will follow the technical approach in~\cite{MeNo05,CKNS08,HTZ09a}, and rely largely on establishing quasi-orthogonality.
Contraction results are established in~\cite{MeNo05,CKNS08} for~\eqref{eqn:strong-eqn1}--\eqref{eqn:strong-b1c} in the case that $A$ is SPD, Lipschitz or piecewise Lipschitz, $b$ is divergence-free, and $c \ge 0$.
In~\cite{HTZ09a}, quasi-orthogonality is used as the basis for establishing contraction of AFEM for two classes of nonlinear problems.
As in these earlier efforts, relying on quasi-orthogonality will require that we assume that the initial mesh is sufficiently fine, and that the solution to the dual problem $a^\ast(w,v) = g(v), ~ g \in L_2(\Omega)$ is sufficiently smooth, e.g. in $H^2_{\loc}(\Omega)$.

Following~\cite{HTZ09a}, the contraction argument developed in this paper will follow from first establishing three preliminary results for two successive AFEM approximations $u_1$ and $u_2$, and then applying the D\"orfler marking strategy:
\begin{enumerateX}

\item Quasi-orthogonality (\S\ref{quasi_orth1o}): There exists $ \Lambda > 1$ such that 
\[
\nrse{u - u_2}^2 \le  \Lambda \nrse{u - u_1}^2 - \nrse{u_2 - u_1}^2.
\]

\item Error estimator as upper bound on error (\S\ref{subsec:uppe1r}): There exists $C_1 > 0$ such that
\[
\nrse{u - u_k}^2 \le C_1 \eta_k^2(u_k, \cT_k), \quad k = 1,2.
\]
\item Estimator reduction (\S\ref{subsec:estim_redu1c}): For $\cM$ the marked set that takes refinement $\cT_1 \goto \cT_2$, for positive constants $\lambda < 1 \an \Lambda_1$ and any $\delta > 0$ 
\[
\eta_2^2(v_2,\cT_2) \le (1 + \delta) \{ \eta_1^2(v_1 , \cT_1) - \lambda \eta_1^2(v_1, \cM \} + 
(1 + \delta^{-1}) \Lambda_1 \eta_0^2 \nrse{v_2 - v_1}.
\]

\end{enumerateX}
The marking strategy used is the original D\" orfler strategy; elements are marked for refinement based on indicators alone.  The marked set $\cM$ must satisfy
\[
\sum_{T \in \cM} \eta_k^2(u_k,T) \ge \theta^2 \eta_k^2(u_k, \cT_k).
\]
In this goal-oriented method, a second marked set is chosen based on an error indicator for the dual problem associated with the given goal functional, and the union of the two marked sets is then used for refinement.
A main advantage of the approach in~\cite{CKNS08,HTZ09a} is that it does not require an interior node property.   This allows us to establish the necessary results for contraction without taking full refinements of the mesh at each iteration. This improvement follows from the use of the local perturbation estimate or local Lipschitz property rather than the estimator as lower bound on error. We use the standard lower bound estimate as found in~\cite{MeNo05} for optimality arguments in the second part of the paper concerning quasi-optimality of the method.

There are three main notions of error used throughout this paper.
The energy error $\nrse{u - u_k}$, the quasi-error, and the total-error.
The \emph{energy error} is defined by the symmetric part of the bilinear form that arises from the given differential operator in~\eqref{primal_proble1m}.
The \emph{quasi-error} is the $l_2$ sum of the energy-error and scaled error estimator
\[
Q_k(u_k, \cT_k) \coloneqq (\nrse{u - u_k}^2 + \gamma \eta_k^2)^{1/2},
\]
and this is the quantity that is reduced at each iteration of the algorithm.  In  \S\ref{sec:contraction_thm1s} the quasi-error  is shown to satisfy
\[
\nrse{u - u_{k+1}}^2 + \gamma \eta_{k+1}^2 \le \alpha^2\left(  \nrse{u - u_k}^2 + \gamma \eta_k^2 \right), ~ \alpha < 1.
\]
 The \emph{total error} includes the oscillation term rather than the estimator
\[
E_k(u_k, \cT_k) \coloneqq (\nrse{u-u_k}^2 + \osc_k^2)^{1/2}.
\]
The oscillation term captures the higher-frequency oscillations in the residual missed by the averaging of the finite element method.
While the quasi-error is the focus of the contraction arguments, the total error is used in our discussion of complexity analysis.

Throughout this paper, the constant $C$ will denote a generic but global constant that may depend on the data and the condition of the initial mesh $\cT_0$, and may change as an argument proceeds, without danger of confusion.

{\bf\em Outline of the paper.}
The remainder of the paper is structured as follows.
In \S\ref{sec:setu1p},
we first describe the problem class and review some standard
facts concerning conforming finite element discretization and
error-estimate-driven adaptive finite element methods (AFEM).
In \S\ref{sec:goafe1m},
we then describe a goal-oriented variation of the standard
approach to AFEM (GOAFEM).
Following the recent work of Mommer and Stevenson for symmetric problems, 
in \S\ref{sec:contraction_thm1s}
we establish contraction of goal-oriented AFEM.
We also then show convergence in \S\ref{subsec:goalConvergenc1e}
in the sense of the goal function.
Our analysis combines the recent
contraction frameworks developed in~\cite{CKNS08,NSV09,HTZ09a}, applied now to the goal oriented problem.
In \S\ref{sec:numerics}, we present some numerical experiments comparing the method presented here with two standard goal oriented strategies.
We recap the results in \S\ref{sec:conc}, and point out some remaining
open problems.

\section{Problem class, discretization, goal-oriented AFEM}
   \label{sec:setu1p}

\subsection{Problem class, weak formulation, spaces and norms}
\label{subsec:problem_dat1a}
\label{subsec:norm1s}
\label{subsec:N1norm1s}

Consider the nonsymmetric problem~\eqref{primal_proble1m}, 
where as in~\eqref{primal_form1} we have
\[
a(u, v) = \langle A \grad u , \grad v \rangle + \langle b \cdot \grad u, v \rangle + \langle cu, v\rangle.
\]
Here we have introduced the notation
$\langle \cdot, \cdot \rangle$ for the $L_2$ inner-product over 
$\Omega \subset \R^d$.  
The adjoint or dual problem is: Find $z \in H_0^1(\Omega) \st$
\begin{equation}\label{dual_proble1m}
a^\ast(z,v) = g(v) ~\tforall v \in H_0^1(\Omega)
\end{equation}
where $\dualforma$ is the formal adjoint of $\forma$,
and where the functional is defined through
\begin{equation}\label{dual_functiona1l}
g(u)= \int_{\Omega} gu ~dx,
\end{equation}
for some given $g \in L_2(\Omega)$.
We will make the following assumptions on the data:
\begin{assumption}[Problem data]\label{data_assumption1s}
The problem data $D= (A,b,c,f)$ and dual problem data $D^\ast = (A,-b,c,g)$ satisfy
\begin{enumerateX}
\item $A: \overline \Omega \goto \R^{d \times d}$,
      Lipschitz, and a.e. symmetric positive-definite:
\begin{align}\label{a1epd1}
\essinf_{x \in \Omega} \lambda_{\text{min}}(A(x)) &= \mu_0 > 0,\\
\label{a1epd2}
\essup_{x \in \Omega} \lambda_{\text{max}}(A(x)) &= \mu_1 < \infty.
\end{align}
\item $b: \overline \Omega \goto \R^d$, with $b_k \in L_\infty(\Omega)$ ,
       and $b$ divergence-free.
\item $c: \overline \Omega \goto \R$, with $c \in L_\infty(\Omega)$, 
      and $c(x) \ge 0 \tforall x \in \Omega$.
\item $f, g \in L_2(\Omega)$.
\end{enumerateX}
\end{assumption}

The native norm is the Sobolev $H^1$ norm given by
\begin{equation}\label{h1_nor1m}
\nr v_{H^1}^2 =  \langle \grad v, \grad v \rangle + \langle v,v \rangle.
\end{equation}
The $L_p$ norm of a vector valued function $v$ over domain $\omega$ is defined here as the $l_2$ norm of the $L_p(\omega)$ norm of each component
\begin{align}\label{L1p_norm_def}
\nr v_{L_p(\omega)} &=  \left(  \sum_{j = 1}^d \left( \int_\omega v_j^p \right)^{2/p} \right)^{1/2},~p = 1, 2, \ldots \nonumber \\
\nr v_{L_\infty(\omega)} &=   \left(  \sum_{j = 1}^d \left( \underset{\omega }\essup  v_j \right)^{2} \right)^{1/2} .
\end{align}
Similarly, the $L_p$ norm of a matrix valued function $M$ over domain $\omega$ is defined as the Frobenius norm of the $L_p(\omega)$ norm of each component
\begin{align}\label{Lp_mat_norm_def}
\nr M_{L_p(\omega)} &=  \left(  \sum_{i,j = 1}^d \left( \int_\omega M_{ij}^p \right)^{2/p} \right)^{1/2} ,  \quad p = 1, 2, \ldots \nonumber \\
\nr M_{L_\infty(\omega)} &=   \left(  \sum_{ij = 1}^d \left( \underset{\omega }\essup  M_{ij} \right)^{2} \right)^{1/2}.
\end{align}
We note that one could employ other equivalent discrete $l_p$ norms in the definitions \eqref{L1p_norm_def} and \eqref{Lp_mat_norm_def}, however this choice simplifies the analysis.

Continuity of $\forma$ follows from the H\"older inequality, and bounding the $L_2$ norm of the function and its gradient by the $H^1$ norm
\begin{align}\label{continuit1y} 
a(u, v)  \le \left( \mu_1 + \nr b_{L_\infty} + \nr c_{L_\infty} \right) \nr u_{H^1} \nr v_{H^1}
 = M_c \nr{u}_{H^1} \nr v_{H^1}.
\end{align}
Coercivity follows from the Poincar\'e inequality with constant $C_\Omega$ and the divergence-free condition
\begin{align}\label{coerciv1e}
a(v,v) & \ge \mu_0|v|_{H^1}^2 
 \ge C_\Omega \mu_0 \nr v_{H^1}^2 = m_{\cE}^2\nr v_{H^1}^2,
\end{align}
where the coercivity constant  $m_\cE^2 \coloneqq C_\Omega \mu_o$. Continuity and coercivity imply existence and uniqueness of the solution by the  Lax-Milgram Theorem~\cite{GT77}.
The adjoint operator $a^\ast(~,~)$ is given by
\[
 a^\ast(v, u) \coloneqq a(u,v), \qquad u,v \in H^1_0(\Omega).
\]
Integration by parts on the convection term and the divergence-free condition imply
\begin{equation}\label{dualop}
a^\ast(z,v) \coloneqq \langle A \grad z, \grad v \rangle - \langle b \cdot \grad z,v \rangle + \langle cz ,v\rangle.
\end{equation}
Define the energy semi-norm by
\begin{equation}\label{energy_seminorm}
\nrse v^2 \coloneqq a(v,v).
\end{equation}
Non-negativity follows directly from the coercivity estimate (\ref{coerciv1e})
\begin{equation}\label{explicit_coerciv1e}
\nrse{v}^2 \ge m_\cE^2 \nr{v}_{H^1}^2,
\end{equation}
which establishes the energy semi-norm as a norm.
Putting this together with the reverse inequality
\begin{equation}\label{norm_continuit1y}
\nrse v^2 \le \mu_1 |\grad v|_{L_2}^2 + \nr c_{L_\infty}\nr{v}_{L_2}^2  \implies \nrse v \le M_\cE \nr v_{H^1},
\end{equation}
establishes the equivalence between the native and energy norms with the constant
$M_\cE = (\mu_1 +  \nr c_{L_\infty} )^{\half}.$

\subsection{Finite element approximation}\label{subsec:mes1h}

We employ a standard conforming piecewise polynomial finite element
approximation below.
\begin{assumption}[Finite element mesh] We make the following assumptions on the underlying simplex mesh:
\begin{enumerateX}
\item The initial mesh $\cT_0$ is conforming.
\item The mesh is refined by newest vertex bisection~\cite{BDD04}, \cite{MoSt09} at each iteration.  
\item The initial mesh $\cT_0$ is sufficiently fine.  In particular, it satisfies (\ref{fine_init_mes1h}). 
\end{enumerateX}
\end{assumption}

Based on assumptions~\ref{mesh_assumption1s}\label{mesh_assumption1s} we have the following mesh constants.
\begin{enumerateX}
\item Define
\begin{equation}\label{hde1f}
h_\cT \coloneqq \max_{T \in \cT} h_T, \quad \where h_T = |T|^{1/d}.
\end{equation}
In particular, $h_0$ is the initial mesh diameter.
\item Define the mesh constant $\gamma_N = 2 \gamma_r$ where
$
\gamma_r = \f {h_{0}}{h_{min}}$ and $h_{min} = \min_{T \in \cT_0} h_T
$
then for any two elements $T, \tilde T$ in the same generation
$
h_T \le \gamma_r h_{\tilde T}$
and as neighboring elements may differ by at most one generation for any two neighboring elements $T \an T'$
\begin{equation}\label{neighborRa1t}
h_{T} \le 2 \gamma_r h_{T'} = \gamma_N h_{T'}.
\end{equation}
\item The minimal angle condition satisfied by newest vertex bisection implies  the meshsize $h_T$ is comparable to $h_\sigma$, the size of any true-hyperface $\sigma$ of $T$.  In particular, there is a constant $ \bar \gamma$
\begin{equation}\label{side_element_com1p}
 \f{h_\sigma}{h_T} \le \bar \gamma^2 \tforall T.
\end{equation}
\end{enumerateX}

Let $\T$ the set of conforming meshes derived from the initial mesh $\cT_0$.  Define $\T_N \subset \T$ by
$
\T_N = \{ \cT \in \T ~\rest~ \#\cT - \# \cT_0 \le N \}.
$
For a conforming mesh $\cT_1$ with a conforming refinement $\cT_2$ we say $\cT_2 \ge \cT_1$.  The set of refined elements is given by
\begin{equation}\label{refineSe1t}
\cR_{1\goto 2} \coloneqq \cR_{\cT_1 \goto \cT_2} \coloneqq \cT_1 \setminus(\cT_2 \cap \cT_1).
\end{equation}
Define the finite element space
\begin{equation}\label{Vta1u}
\V_\cT \coloneqq H_0^1(\Omega) \cap \prod_{T \in \cT} \PP_n(T) \quad \an \V_k \coloneqq \V_{\cT_k}.
\end{equation}
For subsets $\omega \subseteq \cT$,
\begin{equation}\label{Vtau_omeg1a}
\V_\cT(\omega)  \coloneqq H_0^1(\Omega) \cap  \prod_{T \in \omega}  \PP_n(T),
\end{equation}
where $\PP_n(T)$ is the space of polynomials degree degree $n$ over $T$.
Denote the patch about $T \in \cT$
\begin{equation}\label{def:patc1h}
\omega_T \coloneqq T \cup \{ T' \in \cT ~\rest~ T \cap T' \text{ is a true-hyperface of } T \}.
\end{equation}
For a $d$-simplex $T$, an true-hyperface is a $d-1$ dimensional  face of $T$, \iic{e.g., } a face in 3D or an edge in 2D.
Define the discrete primal problem: Find $u_k \in \V_k \st$
\begin{equation}\label{discrete_prima1l}
a(u_k, v_k) = f(v_k), ~ v_k \in \V_k,
\end{equation}
and the discrete dual problem
\begin{equation}\label{discrete_dua1l}
a^\ast( z_k, v_k ) = g(v_k), ~ v_k \in \V_k.
\end{equation}

\subsection{Goal oriented AFEM (GOAFEM)}
   \label{sec:goafe1m}

As in~\cite{MoSt09} the goal oriented adaptive finite element method (GOAFEM) is based on the 
standard AFEM algorithm:
\begin{equation}\label{semr_nsym}
\text{ SOLVE } \rightarrow \text{ ESTIMATE } \rightarrow \text{ MARK } \rightarrow \text{ REFINE }.
\end{equation}
In the goal oriented method, one enforces contraction of the quasi-error in both 
the primal problem and an associated dual problem.
As shown in section \S\ref{subsec:goalConvergenc1e},
the error in the goal-function satisfies the bound
 \begin{align*}
 |g(u) - g(u_k)| = |a(u - u_k, z - z_k)| \le  2\nrse{u - u_k}\nrse{z - z_k}. 
 \end{align*}
This motivates driving down the energy-error in both the primal and 
dual problems at each iteration. As noted in~\cite{CKNS08} the residual-based error estimator does not exhibit monotone behavior in general, although it is monotone non-increasing with respect to nested mesh refinement when applied to the same (coarse) function.  The quasi-error is shown to contract for each problem for which mesh refinement satisfies the D\"orfler property.  However, refining the mesh with respect to the primal problem does not guarantee the quasi-error in the dual problem will be non-increasing, and vice-versa.
As such, the procedures SOLVE and ESTIMATE are performed for each of the 
primal and dual problems.
The marked set is taken to be the union of marked sets from the primal 
and dual problems, each chosen to satisfy the D\"orfler property.
This method produces a sequence of refinements for which the quasi-error 
in the both the primal and dual problems contract at each step.  The requirement to reduce the quasi-error rather than the energy error as in~\cite{MoSt09} is why the marking strategy in this method differs from the one shown effective for the Laplacian.   Our numerical results demonstrate similar behavior of both methods, although the method presented here has the advantage that the code takes fewer iterations of~\eqref{semr_nsym} to achieve similar results.

{\bf\em Procedure SOLVE.}
The contraction result supposes the exact Galerkin solution is found on each mesh refinement.

{\bf\em Procedure ESTIMATE.}
The estimation of the error on each element is determined by a standard residual-based estimator. The residuals over element interiors and jump-residuals over the boundaries are based on
the {local strong forms} of the elliptic operator and its adjoint as follows.
\begin{equation}\label{Lde1f}
\cL(v) = \grad \cdot (A \grad v) - b\cdot \grad v - c  v; 
\quad \cL^\ast(v) = \grad \cdot (A \grad v) + b\cdot \grad v - c  v.
\end{equation}
The {\em residuals} for the primal and dual problems 
using the sign convention in~\cite{CKNS08} are:
\begin{equation}\label{primalres1i} 
R(v) \coloneqq  f + \cL(v); \quad R^\ast(v) \coloneqq  g + \cL^\ast(v), ~v \in \V_\cT.
\end{equation}
While the primal and dual solutions $u \an z$ of~\eqref{primal_proble1m} and~\eqref{dual_proble1m} respectively satisfy
\[
f(z) = a(u,z)  = a^\ast(z,u) = g(u)
\]
the residuals for the primal and dual problems are in general different.
The {\em jump residual} for the primal and dual problems is
\begin{equation}\label{jump_res1i}
J_T(v) \coloneqq \lbb [A \grad v] \cdot n \rbb_{\pa T},
\end{equation}
where {\em jump operator} $\lbb ~\cdot~ \rbb$ is given by
\begin{equation}\label{jump_o1p}
\lbb \phi \rbb_{\pa T} \coloneqq \lim_{t \goto 0} \phi(x + t n) - \phi(x - tn),
\end{equation}
and $n$ is taken to be the appropriate outward normal defined piecewise on $\pa T$.
On boundary edges $\sigma_b$ we have
\[
\lbb [A \grad v] \cdot n \rbb_{\sigma_b} \equiv 0
\]
so that $\lbb [A \grad v] \cdot n \rbb_{\pa T} = \lbb [A \grad v] \cdot n \rbb_{\pa T\cap \Omega}$.
For clarity, we will also employ the notation
\[
R_T(v) \coloneqq R(v) \rest_T, ~v \in \V_\cT,
\]
and similarly for the other strong form operators. 
The error indicator is given as
\begin{equation}\label{eta_cmpc1t}
\eta_\cT^p(v,T) \coloneqq h_T^p \nr{R(v)}_{L_2(T)}^p +  h_T^{p/2} \nr{ J_T(v)  }_{L_2(\pa T)}^p,   \quad v \in \V_\cT.
\end{equation}
The dual error-indicator is then given by
\begin{equation}\label{zeta_cmpc1t}
\zeta_\cT^p(w,T) \coloneqq h_T^p \nr{R^\ast( w)}_{L_2(T)}^p + 
h_T^{p/2} \nr{  J_T( w)   }_{L_2(\pa T)}^p, \quad w \in \V_\cT.
\end{equation}
The error estimators are given by the $l_p$ sum of error indicators over elements in the space where $p = 1$ or $2$.
\begin{equation}\label{primalestimato1r}
\eta_\cT^p(v)   \coloneqq    \sum_{T \in \cT} \eta_{\cT}^p(v,T) , \quad v \in \V_\cT.
\end{equation}
The dual energy estimator is:
\begin{equation}\label{dualestimato1r}
\zeta_\cT^p(w) \coloneqq  \sum_{T \in \cT} \zeta_\cT^p(w ), \quad w \in \V_\cT.
\end{equation}
The contraction results for the quasi-error presented below will be shown to hold for $p = 1,2$ where the error estimator and oscillation are defined in terms of the $l_p$ norm.
While complexity results are shown  only for $p = 2$, the contraction results for $p=1$ are useful for nonlinear problems; see~\cite{HTZ09a}.

For analyzing oscillation, for $v \in \V_\cT$ let $\Pi^2_m$ the orthogonal projector defined by the best $L_2$ approximation in $\PP_m$ over mesh $\cT$ and $P^2_m = I - \Pi^2_m$. 
Define now the oscillation on the elements $T \in \cT$ for the primal problem by
\begin{equation}\label{osc}
\osc_\cT(v,T) \coloneqq h_T \nr{P_{2n-2}^2R(v)}_{L_2(T)} 
\end{equation}
and analogously for the dual problem. For subsets $\omega \subseteq \cT$ set
\begin{equation}\label{osc_set}
\osc_\cT^p(v,\omega) \coloneqq  \sum_{T \in \omega} \osc_\cT^p(v,T).
\end{equation}
The data estimator and data oscillation, identical for both the primal and dual problems, are given by
\begin{align}\label{data_estimato1r}
\eta_\cT^p (D,T) &\coloneqq h_T^p\left(  \nr{\div A}_{L_\infty(T)}^p + h_T^{-p} \nr A_{L_\infty(\omega_T)}^p  + \nr{c}_{L_\infty(T)}^p  + \nr b_{L_\infty(T)}^p \right), \\
\osc_\cT^p(D,T) &\coloneqq  h_T^p \left(  \nr{P^\infty_{n-1}\div A}_{L_\infty(T)}^p + h_T^{-p}\nr{P^\infty_n A}_{L_\infty(T)}^p \right. \nonumber \\
 & \left. + h_T^p \nr{P^\infty_{n-2}c}_{L_\infty(T)}^p + \nr{P^\infty_{2n-2}c}_{L_\infty(T)}^p
 + \nr{P^\infty_{n-1} b}_{L_\infty(T)}^p \right).
\label{data_os1c}
\end{align}
The data estimator and oscillation over the mesh $\cT$ or a subset $\omega \subset \cT$ are given by the maximum data estimator (oscillation) over elements in the mesh or subset: For $\omega \subseteq \cT$
\[
\eta_\cT (D,\omega) = \max_{T \in \omega} \eta_{\cT}(D,T) \an \osc_\cT (D,\omega) = \max_{T \in \omega} \osc_{\cT}(D,T).
\]
The data estimator and data oscillation on the initial mesh 
\[
\eta_0 \coloneqq \eta_{\cT_0} (D,\cT_0),  ~\an~ \osc_0 \coloneqq \osc_{\cT_0} (D,\cT_0).
\]
As the grid is refined, the data estimator and data oscillation terms satisfy the monotonicity property~\cite{CKNS08} for refinements $\cT_2 \ge \cT_1$
\begin{align}\label{mono_dat1a}
\eta_2(D,\cT_2) \le \eta_1(D,\cT_1) ~\an ~ \osc_2(D,\cT_2) \le \osc_1(D,\cT_1).
\end{align}

{\bf\em Procedure MARK.}
The D\"orfler  marking strategy for the goal-oriented problem is based on the following steps as in~\cite{MoSt09}:  
\begin{enumerateX}
\item
Given $\theta \in (0,1)$, mark sets for each of the primal and dual problems:
\begin{itemizeX}
\item Mark a set $\cM_p \subset \cT_k$ such that,
\begin{equation}\label{m1arkP}
\sum_{T \in \cM_p} \eta_k^2(u_k,T) \ge \theta^2 \eta_k^2(u_k, \cT_k)
\end{equation}
\item Mark a set $\cM_d \subset \cT_k$ such that,
\begin{equation}\label{m1arkD}
\sum_{T \in \cM_d} \zeta_k^2(z_k,T) \ge \theta^2 \zeta_k^2(z_k, \cT_k)
\end{equation}
\end{itemizeX}
\item
Let $ \cM = \cM_p \cup \cM_d$ the union of  sets found for the primal and dual problems respectively.  
\end{enumerateX}

The set $\cM$ differs from that in~\cite{MoSt09}, where the set of lesser cardinality between $\cM_p \an \cM_d$ is used.  In the case of the nonsymmetric problem the error reduced at each iteration is the quasi-error rather than the energy error as in the symmetric problem~\cite{MoSt09}.  This error for each problem is guaranteed to contract based on the refinement satisfying the D\"orfler property. 
As such, refining the mesh with respect to one problem does not guarantee the quasi-error in the other problem is nonincreasing. Sets $\cM_p \an \cM_d$ with optimal cardinality (up to a factor of 2) can be chosen in linear time~ by binning the elements rather than performing a full sort~\cite{MoSt09}.

{\bf\em Procedure REFINE.}
The refinement (including the completion) is performed according to newest vertex bisection~\cite{BDD04}.
The complexity and other properties of this procedure are now well-understood, and will simply be exploited here.

\section{Contraction and convergence theorems}
   \label{sec:contraction_thm1s}

The key elements of the main contraction argument constructed below are quasi-orthogonality~\ref{subsec:quasi_ortho1g}, error estimator as upper-bound on energy-norm error~\ref{subsec:uppe1r} and estimator reduction~\ref{subsec:estim_redu1c}.
Estimator-reduction is shown via the local-perturbation estimate~\ref{subsec:local_pertur1b}. The local perturbation of the oscillation is presented here and used in \S\ref{sec:complexity}.
Mesh refinements $\cT_1 \an \cT_2$ (respectively $\cT_j$) are assumed conforming, and $u_j$ is assumed the Galerkin solution on refinement $\cT_j$.  The following results hold for both the primal and dual problems which differ by the sign of the convection term; therefore, they are established here only for the primal problem.

\subsection{Quasi-orthogonality }\label{subsec:quasi_ortho1g}

Orthogonality in the energy-norm $\nrse{u - u_{2}}^2= \nrse{u - u_1}^2 - \nrse{u_{2} - u_1}^2$ does not generally hold in the nonsymmetric problem.
We use the weaker quasi-orthogonality result to establish contraction of AFEM (GOAFEM).
The following is a variation on Lemma~2.1 in~\cite{MeNo05} (see also~\cite{HTZ09a}).
\begin{lemma}[Quasi-orthogonality]

 Let the problem data satisfy Assumption~\ref{data_assumption1s} and the mesh satisfy conditions (1) and (2) of Assumption~\ref{mesh_assumption1s}. Let $\cT_1, \cT_2 \in \T$ with $\cT_2 \ge \cT_1$. Let $u_k \in \V_k$ the solution to~\eqref{discrete_prima1l}, $k = 1,2$.  There exists a constant $C_\ast > 0$ depending on the problem data $D$ and initial mesh $\cT_0$, and a number $0 < s \le 1$ dictated only by the angles of $\pa \Omega$, such that if the meshsize $h_0$ of the initial mesh satisfies $\bar \Lambda \coloneqq C_\ast h_0^s \nr b_{L_\infty} \mu_0^{-1/2}< 1$, then

\begin{equation}\label{quasi_orth1o}
\nrse{u - u_{2}}^2 \le \Lambda \nrse{u - u_1}^2 - \nrse{u_{2} - u_1}^2,
\end{equation}
where
\[
\Lambda \coloneqq (1 - C_\ast h_0^s \nr b_{L_\infty} \mu_0^{-1/2})^{-1}.
\]
Equality holds (usual orthogonality) when $b = 0$ in $\Omega$, in which case the problem is symmetric.
\end{lemma}

\begin{proof}
The proof follows close that of Lemma~2.1 in~\cite{MeNo05}.
Let
\[
e_{2} \coloneqq u - u_{2}, \quad e_1 \coloneqq  u - u_{1}, \quad  \an \vareps_1  \coloneqq u_{2} - u_1.
\]
By Galerkin orthogonality
\begin{align} \label{error_not_orth1o} 
\nrse{e_1}^2  = a(e_1, e_1)  = \nrse{e_{2}}^2 + \nrse{\vareps_1}^2 + a(\vareps_1, e_{2}).
\end{align}
Rearranging and applying the divergence-free condition on the convection term
\[
\nrse{e_{2}}^2 = \nrse{e_1}^2 - \nrse{\vareps_1}^2 - 2 \langle b \cdot \grad \vareps_{1}, e_{2}\rangle .
\]
Applying H\"older's inequality and coercivity (\ref{coerciv1e}) $|\vareps_1|_{H^1} \le  \mu_0^{-\half} \nrse {\vareps_1}$ followed by Young's inequality with constant $\delta$ to be determined,
\begin{align}\label{2b_youn1g} 
- 2 \langle b \cdot \grad \vareps_{1}, e_{2} \rangle
 \le \delta \nr{e_{2}}_{L_2}^2 +  \f{ \nr b_{L_\infty}^2}{\delta \mu_0 } \nrse{  \vareps_{1}} ^2.
\end{align}
By a duality argument for some $C_\ast>0$ assuming 
$u\in H^{1+s}(\Omega)$ for some $0 < s \le 1$ depending on the angles of $\pa \Omega$
\begin{equation}\label{L2_le_energ1y}
\nr{e_{2}}_{L_2} \le  C_\ast h_0^s \nrse {e_{2}}.
\end{equation}
The details of this argument as described in the appendix~\S\ref{subsec:duality} may also be found in~\cite{AOB01} and~\cite{Ci02}.
Applying (\ref{L2_le_energ1y}) and (\ref{2b_youn1g}) to (\ref{error_not_orth1o}),
\begin{equation}
(1 - \delta C_\ast^2 h_0^{2s})\nrse{u - u_{2}}^2  \le \nrse{u - u_1}^2 -\left( 1 - \f {\nr b_{L_\infty}^2}{\delta \mu_0} \right) \nrse{u_1 - u_{2}}^2.
\end{equation}
Choose $\delta$ to equate coefficients
\[
\delta C_\ast^2 h_0^{2s} =\f {\nr b_{L_\infty}^2}{\delta \mu_0} \implies \delta  = \f{\nr b_{L_\infty}}{C_\ast h_0^s\sqrt{\mu_0}},
\]
then 
\[
\nrse{u - u_{2}}^2 \le \left( 1 - \nr b_{L_\infty} C_\ast h_0^s \mu_0^{-\half} \right)^{-1}\nrse{u - u_1}^2 - \nrse{u_1 - u_{2}}^2.
\]
Assuming the initial mesh as characterized by $h_0$ satisfies
\begin{equation}\label{fine_init_mes1h}
\bar \Lambda = \nr b_{L_\infty} C_\ast h_0^s \mu_0^{-\half}  < 1,
\end{equation}
the quasi-orthogonality result holds.
\end{proof}

Note that by (\ref{error_not_orth1o}) we also have
\begin{equation}
\nrse{\vareps_1}^2 = \nrse{e_1}^2 - \nrse{ e_2}^2 - 2\langle b \cdot \grad e_2 , \vareps_1 \rangle.
\end{equation}
Similarly to (\ref{2b_youn1g})
\begin{align}
- 2 \langle b \cdot \grad e_{2}, \vareps_1\rangle  \ge -2| \langle b \cdot \grad e_{2}, \vareps_1 \rangle |
 \ge -\delta \nr{ \vareps_1}_{L_2}^2 -  \f{ \nr b_{L_\infty}^2}{\delta \mu_0 } \nrse{  e_{2}}^2, 
\end{align}
which under the same assumptions yields the estimate
\begin{equation}\label{other_orth1o}
\nrse{u_{2} - u_1}^2 \ge (1 + \bar \Lambda)^{-1} \nrse{u - u_1}^2 - \nrse{ u - u_{2}}^2,
\end{equation}
where $\bar \Lambda < 1 \implies (1 + \bar \Lambda)^{-1} > 1/2$. 

\subsection{Error estimator as global upper-bound}\label{subsec:uppe1r}

We now recall the property that the error estimator is a global upper bound on the error.
The proof is fairly standard; see e.g.~\cite{MoSt09} (Proposition 4.1), \cite{MeNo05} (3.6), and~\cite{HTZ09a}.
\begin{lemma}[Error estimator as global upper-bound]
Let the problem data satisfy Assumption~\ref{data_assumption1s} and the mesh satisfy conditions (1) and (2) of Assumption~\ref{mesh_assumption1s}. Let $\cT_1, \cT_2 \in \T$ with $\cT_2 \ge \cT_1$.  Let $u_k \in \V_k$ the solution to~\eqref{discrete_prima1l}, $k = 1,2$  and $u$ the solution to~\eqref{primal_proble1m}. Let 
\[
G = G(\cT_2, \cT_1) \coloneqq \{T \subset \cT_1 ~\rest~T \cap \tilde T \ne \emptyset \text{ for some } \tilde T \in \cT_1, \tilde T \notin \cT_2\}.
\] 
Then for global constant $C_1$ depending on the problem data $D$ and initial mesh $\cT_0$
\begin{equation}\label{indicators_ge_energ1y}
\nrse{u_{2} - u_1} \le C_1   \eta_1(u_1,G) 
\end{equation}

and in particular 
\begin{equation}\label{estimator_ge_erro1r}
\nrse{u - u_1} \le C_1 \eta_1(u_1, \cT_1).
\end{equation}
\end{lemma}

\subsection{Local perturbation}\label{subsec:local_pertur1b}
The local perturbation property established in~\cite{CKNS08}, analogous to the local Lipshitz property in~\cite{HTZ09a}, is a key step in establishing the contraction result.  This is a minor variation on Proposition 3.3 in~\cite{CKNS08} which deals with a symmetric problem.  Here, we include a convection term in the estimate.
In particular, \eqref{perturb_es1t} shows that the difference in the error indicators over an element $T$ between two functions in a given finite element space may be bounded by a fixed factor of the native norm over the patch $\omega_T$ of the difference in functions.
In contrast with the analogous result in~\cite{CKNS08} the estimate~\eqref{perturb_os1c} involves a fixed factor of the native norm over an individual element rather than a patch as by the continuity of $A$ the oscillation term does not involve the jump residual.

We include the proof of ~\eqref{perturb_es1t} for completeness.  The proof of~\eqref{perturb_os1c} may be found in~\cite{CKNS08} with the final result inferred by the absence of the jump residual in the oscillation term.

\begin{lemma}[Local perturbation]
Let the problem data satisfy Assumption~\ref{data_assumption1s} and the mesh satisfy condition (1) of Assumption~\ref{mesh_assumption1s}. Let $\cT \in \T$. For all $T \in \cT$ and for any  $v, w \in \V_\cT$
\begin{align}\label{perturb_es1t}
\eta_\cT(v,T) & \le \eta_\cT(w,T) + \bar \Lambda_1 \eta_\cT(D,T)\nr{v - w}_{H^1(\omega_T)} \\
\label{perturb_os1c}
\tosc_\cT(v,T) & \le \tosc_\cT(w,T) + \bar \Lambda_2 \tosc_\cT(D,T)\nr{v - w}_{H^1(T)}
\end{align}
where recalling~\eqref{def:patc1h}  $\omega_T$ is the union of $T$ with elements in $\cT$ sharing a true-hyperface with $T$.  The constants $\bar \Lambda_1, \bar \Lambda_2 > 0$ depend on the initial mesh $\cT_0$, the dimension $d$ and the polynomial degree $n$.
\end{lemma}

\begin{proof}[Proof of~\eqref{perturb_es1t}]
From~\eqref{eta_cmpc1t}
\begin{equation}
\eta_\cT^p(v,T) \coloneqq h_T^p \nr{R(v)}_{L_2(T)}^p +  h_T^{p/2} \nr{ J_T(v)    }_{L_2(\pa T)}^p,   \quad v \in \V_\cT.
\end{equation}
Denote $\eta_\cT(v,T)$ by $\eta(v,T)$. Set $e = v - w$. By linearity
\begin{align*}
R(v) = R(w + e) = f + \cL( w + e) = f + \cL(w) + \cL(e) = R(w) + \cL(e)
\end{align*}
and
\begin{align*}
J(v) = J(w + e) = J(w) + J(e).
\end{align*}
 For $p = 1$ by the triangle inequality
\begin{align*}
\eta(v,T) & = h_T \nr{R(w) + \cL(e)}_{L_2(T)} + h_T^\half \nr{J(w) +J(e)}_{L_2(\pa T)} \\
& \le \eta(w,T) + h_T\nr{\cL (e)}_{L_2(T)} + h_T^\half \nr{J(e)}_{L_2(\pa T)} .
\end{align*}
For $p = 2$ using the generalized triangle-inequality
\begin{equation}\label{subsec:triangl1e}
\sqrt{(a + b)^2 + (c + d)^2} \le \sqrt{a^2 + c^2} + b + d, \quad \text{ for }  a,b,c,d > 0
\end{equation}
we have
\begin{align*}
\eta(v,T) & = \left( h_T^2 \nr{R(w) + \cL(e)}_{L_2(T)}^2 + h_T \nr{J(w) + J(e)}_{L_2(\pa T)}^2\right)^{1/2} \\
& \le \eta(w,T) + h_T\nr{\cL (e)}_{L_2(T)} + h_T^\half \nr{J(e)}_{L_2(\pa T)} .
\end{align*}
Consider the second term on the RHS $h_T\nr{\cL(e)}_{L_2(T)}$.  By definition~\eqref{Lde1f} of $\cL(~\cdot~)$, the product rule applied to the diffusion term and the triangle-inequality
\begin{align*}
\nr{\cL(e)}_{L_2(T)} 
 \le \nr{\div A \cdot \grad e}_{L_2(T)} + \nr{A : D^2 e}_{L_2(T)}  + \nr{ce}_{L_2(T)} + \nr{b \cdot \grad e}_{L_2(T)} 
\end{align*}
where $D^2 e$ is the Hessian of $e$.  Consider each term.
The first diffusion term
\begin{align}\label{PI:diff1}
\nr{\div A \cdot \grad e}_{L_2(T)} & \le  \nr{\div A}_{L_\infty(T)} \nr{\grad e}_{L_2(T)}
\end{align}
by the inequality
\begin{equation}\label{L2toLin1f}
\nr{v \cdot z}_{L_2(T)} \le  \nr{v}_{L_\infty(T)} \nr{z}_{L_2(T)}, \quad v \in L_\infty(T),~ z \in L_2(T).
\end{equation}
Applying~\eqref{L2toLin1f} and inverse-estimate~\cite{BS08} to the second diffusion term
\begin{align}\notag
\nr{A : D^2 e}_{L_2(T)}& \le  \nr {A}_{L_\infty(T)} \nr{D^2 e}_{L_2(T)} \\ \label{PI:diff2}
& \le   C_I h_T^{-1 } \nr {A}_{L_\infty(T)}  \nr{\grad e}_{L_2(T)}.
\end{align}
For the reaction term
\begin{align}\label{PI:react}
\nr{ce}_{L_2(T)} \le \nr c_{L_\infty(T)} \nr e_{L_2(T)}.
\end{align}
For the convection term applying~\eqref{L2toLin1f}
\begin{align}\label{PI:convect}
\nr{b \cdot \grad e}_{L_2(T)} & \le \nr b_{L_\infty(T)} \nr{\grad e}_{L_2(T)}.
\end{align}
Consider the the jump-residual term $\nr{J(e)}_{L_2(\pa T)}$.
For each interior true-hyperface 
$\sigma = T \cap T', ~ T, T' \in \cT$ by~\eqref{jump_o1p}
\begin{align}\label{jump_convention}
J(e) \rest_\sigma  & \coloneqq \lim_{t \goto 0^+}(A \grad e)(x + t n_\sigma) - \lim_{t \goto 0^-}(A \grad e)(x - tn_\sigma) \nonumber \\
& ~=  n_\sigma \cdot (A \grad e)\rest_T - n_\sigma \cdot (A \grad e)\rest_{T'}
\end{align}
where $(A \grad e) \rest_T$ is understood to refer to the product of the limiting value of  $A  \grad e$ as the element boundary is approached from the interior of $T$.  
By the triangle-inequality
\begin{align*}
\nr{J(e)}_{L_2(\sigma)} & \le \nr{n_\sigma \cdot (A \grad e)\rest_T}_{L_2(\sigma)} + \nr{n_\sigma \cdot (A \grad e)\rest_{T'}}_{L_2(\sigma)}.
\end{align*}
By bounds for the inner-product with a unit normal and a matrix-vector product
\begin{align}\label{normal_boun1d}
\nr{\phi \cdot n}_{L_2(\sigma)} & \le \nr \phi_{L_2(\sigma)}, \quad \phi \in L_2(\sigma),\\
\label{MLinfL2}
\nr{M\phi}_{L_2(T)} & \le  \nr{M}_{L_\infty(T)} \nr \phi_{L_2(T)}, \quad M \in L_\infty(T), ~\phi \in L_2(T)
\end{align}
obtain
\begin{align} \label{jr1:split}
\nr{n_\sigma \cdot (A \grad e)\rest_T}_{L_2(\sigma)} & \le \nr{ (A \grad e)\rest_T}_{L_2(\sigma)} 
 \le  \nr {A\rest_T}_{L_\infty(\sigma)} \nr{\grad e \rest_T}_{L_2(\sigma)}.
\end{align}
Applying the trace theorem and an inverse inequality to $\nr{\grad e \rest_T}_{L_2(\sigma)}$ via the inequality
\begin{equation}\label{L2trace}
\nr{\phi}_{L_2(\sigma)} \le C h_T^{-\half} \nr{\phi}_{L_2(T)}, \quad \phi \in L_2(T)
\end{equation}
we have
\begin{equation}\label{jr1:1}
 \nr {\grad e\rest_T}_{L_2(\sigma)}  \le C_T(\bar \gamma)^{d-1} h_T^{-\half}\nr {\grad e}_{L_2(T)}.
\end{equation}
By the Lipschitz property of $A$ \label{Linf_trace}
\begin{equation}\label{jr1:2}
\nr{A\rest_T}_{L_\infty(\sigma)} = \nr{A}_{L_\infty(\sigma)} \le \nr{A}_{L_\infty(T)}.
\end{equation}
By~\eqref{jr1:split}, \eqref{jr1:1}, \eqref{jr1:2} and comparability of mesh diameters~\eqref{neighborRa1t}
\begin{align*}
\nr{J(e)}_{L_2(\sigma)} 
& \le  2C_T(\bar \gamma)^{d-1}  \gamma_N^\half  h_{T}^{-\half} \nr A_{L_\infty(\omega_T)} \nr{\grad e}_{L_2(\omega_T)}.
\end{align*}
Element $T$ has at most $d+1$ interior true-hyperfaces yielding
\begin{align*}
\nr{J(e)}_{L_2(\pa T)} & \le 2 (d+1)~ C_T(\bar \gamma)^{d-1} \gamma_N^\half  h_{T}^{-\half}  \nr A_{L_\infty(\omega_T)} \nr{\grad e}_{L_2(\omega_T)}\\
& = C_Jh_{T}^{-\half} \nr A_{L_\infty(\omega_T)} \nr{\grad e}_{L_2(\omega_T)}.
\end{align*}
Putting together the terms from $\cL$ and from the jump residual,
\begin{align*}
\eta(v,T) & \le \eta(w,T) + h_T \left(  \nr{\div A}_{L_\infty(T)} +  C_I h_T^{-1} \nr{A}_{L_\infty(T)} \right. \\
& + \left.  \nr c_{L_\infty(T)} +  \nr b_{L_\infty(\omega)} \right)\nr{e}_{H^1(T)} + h_T^\half C_{J}h_T^{-\half} \nr A_{L_\infty(\omega_T)} \nr{e}_{H^1(\omega_T)} \\
 & \le \eta(w,T) + C_{TOT'}\, \eta_T(D,T)\nr{v - w}_{H^1(\omega_T)}
\end{align*}
where $C_{TOT'}$ differs by a factor of $2$ for $p = 1,2$.
\end{proof}

\subsection{Estimator reduction}\label{subsec:estim_redu1c}

We now establish one of the three key results we need, namely estimator reduction.  This result is a minor variation of~\cite{CKNS08} Corollary 2.4 and is stated here for completeness.  
\begin{theorem}[Estimator reduction]
Let the problem data satisfy Assumption~\ref{data_assumption1s} and the mesh satisfy conditions (1) and (2) of Assumption~\ref{mesh_assumption1s}. Let $\cT_1 \in \T, ~ \cM \subset \cT_1$ and $\cT_2 = \text{REFINE}(\cT_1, \cM)$.  For $p = 1$ let
\[
\Lambda_1 \coloneqq (d+2)^2\bar \Lambda_1^2 m_\cE^{-2} \quad \an~ \lambda  \coloneqq (1 - 2^{-1/2d})^2 > 0
\]
and for $p = 2$ let 
\[
\Lambda_1 \coloneqq (d+2)\bar \Lambda_1^2 m_\cE^{-2} \quad \an~  \lambda \coloneqq 1 - 2^{-1/d}> 0
\]
with $\bar \Lambda_1$ from~\ref{subsec:local_pertur1b} (Local Perturbation). Then  for any $v_1 \in \V_1$ and $v_2 \in \V_2$ and $\delta > 0$ 
\begin{align}\label{est_redu1c}
\eta_2^2(v_2, \cT_2)  \le  &(1 + \delta) \left\{ \eta_1^2(v_1,\cT_1) - \lambda \eta_1^2(v_1,\cM) \right\}  + ( 1 + \delta^{-1}) \Lambda_1 \eta_0^2 \nrse{v_2 - v_1}^2.
\end{align}
\end{theorem}

\begin{proof}
The proofs for $p = 1$ and $p = 2$ are similar.
For $p = 1$ it is necessary to sum over elements before squaring and 
for $p = 2$ square first then sum over elements.

{\em Proof for the case $p = 1$.}
By the local Lipschitz property~\eqref{perturb_es1t}
\begin{align}\label{e1st_red1}
\eta_2 (v_2, T) \le \eta_2(v_1,T) + \bar \Lambda_1 \eta_2(D,T) \nr{v_2 - v_1}_{H^1(\omega_T)}.
\end{align}
Summing over all elements $T \in \cT_2$, the sum of norms over $\omega_T$ covers each element  at most $(d+2)$ times as each patch $\omega_T$ is the union of element $T$ and the (up to) $d+1$ elements sharing a true-hyperface with $T$.  Then by the coercivity (\ref{explicit_coerciv1e}) over $\Omega$ 
\begin{align}\label{e1st_red2} 
\eta_2 (v_2, \cT_2) 
& \le  \eta_2(v_1,\cT_2) + (d + 2)\bar \Lambda_1 {m_\cE}^{-1}\eta_2^2(D,\cT_2) \nrse{v_2 - v_1}. 
\end{align}
Squaring (\ref{e1st_red2}) and applying Young's inequality with constant $\delta$ to the cross-term,
\begin{align}\label{e1st_red3} \notag %
\eta_2^2 (v_2, \cT_2) & \le (1 + \delta) \eta_2^2(v_1,\cT_2) +(1 + \delta^{-1})(d+2)^2 \bar \Lambda_1^2 m_\cE^{-2} \eta_2^2(D,\cT_2) \nrse{v_2 - v_1}^2\\
& =   (1 + \delta) \eta_2^2(v_1,\cT_2) +(1 + \delta^{-1})\Lambda_1 \eta_2^2(D,\cT_2) \nrse{v_2 - v_1}^2.
\end{align}

For an element $T \in \cM$ marked for refinement, let $\cT_{2,T} \coloneqq \{ T' \in \cT_2 ~\rest~ T' \subset T\}$. As $v_1 \in \V_1$  has no discontinuities across element boundaries in $\cT_{2,T}$, we have $J(v_1) = 0$ on true hyperfaces in the interior of $\cT_{2,T}$.

Recall the element diameter $h_T = |T|^{1/d}$.  For an element $T$ marked for refinement, $T'$ must be a proper subset of $T$, in particular a product of at least one bisection so that
\begin{equation}\label{Tprime_vsT}
|T'| \le \f 1 {2}|T| \leftrightarrow |T'|^{1/d} \le \f 1 {2^{1/d}}|T|^{1/d} \leftrightarrow h_{T'} \le \f 1 {2^{1/d}} h_T. 
\end{equation}
Then 
\begin{align}\label{e1st_red4} \notag %
\sum_{T' \in \cT_{2,T}}\eta_2(v_1,T') & \le  \sum_{T' \in \cT_{2,T}}  h_{T'} \nr{R(v_1)}_{L_2(T')} + \sum_{T' \in \cT_{2,T}} h_{T'}^\half \nr{J(v)}_{L_2(\pa T' \cap \pa T)}\\
\notag %
& \le 2^{-1/d} h_T \sum_{T' \in \cT_{2,T}}\left(   \nr{R(v_1)}_{L_2(T')}  \right)+ 
 2^{-1/2d}h_T^\half  \nr{J(v)}_{L_2(\pa T)}\\
\notag %
& \le  2^{-1/2d} \left(  h_T \nr{R(v_1)}_{L_2(T)} + h_T^\half  \nr{J(v)}_{L_2(\pa T)} \right)\\
& = 2^{-1/2d}\eta_1(v_1,T).
\end{align}
For an element $T \notin \cM$, that is $T' = T$ the indicator is reproduced 
\begin{align}\label{e1st_red5}
\eta_2(v_1,T')  = \eta_1(v_1,T). 
\end{align}
Sum over all $T \in \cT_2$ by estimates (\ref{e1st_red4}),  (\ref{e1st_red5}) writing the sum of  indicators over the  $\cT_1 \setminus \cM$ as the total estimator less the  indicators over the refinement set $\cM$. Let the refined set $\cR \coloneqq \{ T \in \cT_2 ~\rest~ T' \subset \tilde T \text{ for some } \tilde  T \in \cM\}$ then 
\begin{align}\label{e1st_red6} \notag
\eta_2(v_1,\cT_2) & = \sum_{T \in \cT_2} \eta_2(v_1, T) \\
\notag %
& = \sum_{T \in \cT_2 \setminus \cR} \eta_2(v_1,T) +  \sum_{T \in \cR} \eta_2(v_1, T)  \\
\notag %
& \le \eta_1(v_1, \cT_1) - \eta_1(v_1, \cM) + 2^{-1/2d} \eta_1(v_1, \cM) \\
& = \eta_1(v_1, \cT_1) - \lambda_1 \, \eta_1(v_1, \cM)
\end{align}
where $\lambda_1 = 1 - 2^{-1/2d} < 1$.
Squaring~\eqref{e1st_red6}
 \begin{align} \label{e1st_red7} \notag
\eta_2^2(v_1,\cT_2) 
& \le \eta_1^2(v_1, \cT_1) + \lambda_1^2 \, \eta_1^2(v_1, \cM) - 2   \lambda_1^2 \, \eta_1^2(v_1, \cM) \\
& = \eta_1^2(v_1, \cT_1) - \lambda  \, \eta_1^2(v_1, \cM)
\end{align}
 where $\lambda = \lambda_1^2 = (1-2^{-1/2d})^2$.
Applying~\eqref{e1st_red7} to~\eqref{e1st_red3} and applying monotonicity of the data-estimator
\begin{align*}
\eta_2^2 (v_2, \cT_2) 
 & \le  (1 + \delta) \left( \eta_1^2(v_1, \cT_1) - \lambda \, \eta_1^2(v_1, \cM)\right) \\
 & + (1 + \delta^{-1}) \Lambda_1^2 \eta_0^2(D,\cT_0) \nrse{v_2 - v_1}^2. 
\end{align*}

The proof for the  case $p = 2$ is similar and may be found in~\cite{CKNS08}.
\end{proof}

\subsection{Contraction of AFEM}\label{subsec:contractio1n}

We now establish the main contraction results. The contraction result~\ref{contraction_eac1h} is a modification of~\cite{CKNS08} Theorem 4.1.  Here we use quasi-orthogonality to establish contraction of each of the nonsymmetric problems~\eqref{primal_proble1m} and~\eqref{dual_proble1m}.

\begin{theorem}[GOAFEM contraction]\label{contraction_eac1h}
Let the problem data satisfy Assumption~\ref{data_assumption1s} and the mesh satisfy Assumption~\ref{mesh_assumption1s}. Let $u$ the solution to~\eqref{primal_proble1m}. 
Let $\theta \in (0,1]$, and let $\{\cT_k, \V_k, u_k\}_{k \ge 0}$ be the sequence of meshes, finite element spaces and discrete solutions produced by GOAFEM.  Then there exist constants $\gamma > 0 \an 0 < \alpha < 1$, depending on the initial mesh $\cT_0$ and marking parameter $\theta$ such that
\begin{equation}\label{contractionre1s}
\nrse{u - u_{k+1}}^2 + \gamma \eta_{k+1}^2 \le \alpha^2\left(  \nrse{u - u_k}^2 + \gamma \eta_k^2 \right).
\end{equation}
The analogous result holds for the dual problem with $\{\cT_k, \V_k, z_k\}_{k \ge 0}$  the sequence of meshes, finite element spaces and discrete solutions produced by GOAFEM.
\end{theorem}

\begin{proof}
Denote
\[
e_k = u - u_k, \quad e_{k+1} = u - u_{k+1} \quad \an \quad \vareps_k = u_{k+1} - u_k.
\]
Let
\[
\eta_k = \eta_k(u_k, \cT_k),\quad \eta_k(\cM_k) = \eta_k(u_k, \cM_k) \quad \an \quad \eta_{k+1} = \eta_{k+1}(u_{k+1}, \cT_{k+1}).
\]
 By the result of  Estimator Reduction~\ref{subsec:estim_redu1c}, for any $\delta > 0$
\[
\eta^2_{k+1} \le (1 + \delta) \left\{ \eta_k^2 - \lambda \eta_k^2(\cM_k)\right\} + (1 + \delta^{-1}) \Lambda_1 \eta_0^2 \nrse{\vareps_k}^2.
\]
Multiplying this inequality by positive constant $\gamma$ (to be determined) and adding the quasi-orthogonality estimate $\nrse{e_{k+1}}^2 \le \Lambda \nrse{e_k}^2 - \nrse{\vareps_k}^2$ obtain
\begin{align}\label{est41_1}
\nrse{e_{k+1}}^2  + \gamma \eta^2_{k+1} &\le \Lambda \nrse{e_k}^2 - \nrse{\vareps_k}^2 + \gamma  (1 + \delta) \left\{ \eta_k^2 - \lambda \eta_k^2(\cM_k)\right\} \nonumber \\
& \qquad + \gamma (1 + \delta^{-1}) \Lambda_1 \eta_0^2 \nrse{\vareps_k}^2.
\end{align}
Choose $\gamma$ to eliminate $\nrse{\vareps_k}$ the error between consecutive estimates by setting
\begin{align}\label{4_1:choose_gamma}
\gamma (1 + \delta^{-1}) \Lambda_1 \eta_0^2 = 1 \iff \gamma = \f{1}{(1 + 1/\delta)  \Lambda_1 \eta_0^2} \iff \gamma(1 + \delta) = \f {\delta}{\Lambda_1 \eta_0^2}.
\end{align}
Applying~\eqref{4_1:choose_gamma} to~\eqref{est41_1} obtain
\begin{align}\label{est41_2}
\nrse{e_{k+1}}^2  + \gamma \eta^2_{k+1} \le \Lambda \nrse{e_k}^2  + \gamma  (1 + \delta)  \eta_k^2 -  \gamma  (1 + \delta) \lambda \eta_k^2(\cM_k) .
\end{align}
By the D\" orfler marking strategy $\eta_k^2(\cM_k) \ge \theta^2 \eta_k^2$ so that
\begin{align}\label{est41_3}
\nrse{e_{k+1}}^2  + \gamma \eta^2_{k+1} \le \Lambda \nrse{e_k}^2  + \gamma  (1 + \delta)  \eta_k^2 -  \gamma  (1 + \delta) \lambda \theta^2 \eta_k^2 .
\end{align}
Split the last term by factors of $\beta \an (1 - \beta)$ for any $\beta \in (0,1)$ to arrive at
\begin{align}\label{est41_3b}
\nrse{e_{k+1}}^2  + \gamma \eta^2_{k+1} &\le \Lambda \nrse{e_k}^2  + \gamma  (1 + \delta)  \eta_k^2 - \beta \gamma  (1 + \delta) \lambda \theta^2 \eta_k^2 \nonumber \\
& \qquad - (1 - \beta) \gamma  (1 + \delta) \lambda \theta^2 \eta_k^2 .
\end{align}
Applying the upper-bound estimate~\eqref{estimator_ge_erro1r} $\nrse{e_k}^2 \le C_1 \eta_k^2$ to the term multiplied by $\beta$ then by (\ref{4_1:choose_gamma}) 
\begin{align}\label{est41_3c}
\nrse{e_{k+1}}^2  + \gamma \eta^2_{k+1} & \le \Lambda \nrse{e_k}^2  - \f {\beta \gamma  (1 + \delta) \lambda \theta^2}{C_1} \nrse{e_k}^2  + \gamma  (1 + \delta)  \eta_k^2 \nonumber \\
& \qquad - (1 - \beta) \gamma  (1 + \delta) \lambda \theta^2 \eta_k^2 \\
& =  \Lambda \nrse{e_k}^2  - \beta \f { \delta \lambda \theta^2}{C_1 \Lambda_1 \eta_0^2} \nrse{e_k}^2  + \gamma  (1 + \delta)  \eta_k^2 \nonumber \\
& \qquad - (1 - \beta) \gamma  (1 + \delta) \lambda \theta^2 \eta_k^2 \\
& =  \left( \Lambda   - \beta \f { \delta \lambda \theta^2}{C_1 \Lambda_1 \eta_0^2}  \right) \nrse{e_k}^2  + \gamma  (1 + \delta) \left(1   - (1 - \beta)  \lambda \theta^2\right) \eta_k^2  \\
& = \alpha_1 ^2(\delta, \beta)  \nrse{e_k}^2 + \gamma \alpha_2^2(\delta, \beta) \eta_k^2
\end{align}
where
\begin{align}\label{41_def_alpha12}
\alpha_1^2(\delta, \beta) \coloneqq  \Lambda   - \beta \f {  \lambda \theta^2}{C_1 \Lambda_1 \eta_0^2}\delta , \quad \alpha_2^2(\delta, \beta) \coloneqq  (1 + \delta) \left(1   - (1 - \beta)  \lambda \theta^2\right).
\end{align}
Choose $\delta$ small enough so that
\begin{align*}\label{41_deltacond}
\alpha^2 \coloneqq \max\{ \alpha_1^2, \alpha_2^2\} < 1.
\end{align*} 
To ensure such a $\delta$ exists in light of the quasi-orthogonality constant $\Lambda > 1$ observe
\[
\alpha_1^2 < 1 \text{ when } \delta > (\Lambda -1) \f {C_1 \Lambda_1 \eta_0^2}{\beta \lambda \theta^2}
\]
and
\[
\alpha_2^2 < 1 \text{ when } \delta < \left( 1 - (1 - \beta)\lambda \theta^2 \right)^{-1} -1 = 
\f{(1 - \beta) \lambda \theta^2}{ 1 - (1 - \beta)\lambda \theta^2}
\]
so to obtain an interval of positive measure where $\delta$ may be found we require
\begin{align*}
 (\Lambda -1) \f {C_1 \Lambda_1 \eta_0^2}{\beta \lambda \theta^2} < \f{(1 - \beta) \lambda \theta^2}{ 1 - (1 - \beta)\lambda \theta^2}
\end{align*}
placing a second constraint on the quasi-orthogonality constant
\begin{equation}\label{second_QC}
\Lambda < 1 + \f {\lambda^2 \theta^4 \beta(1 - \beta)}{ C_1 \Lambda_1 \eta_0^2 \left( 1 - (1 - \beta)\lambda \theta^2 \right)}
\end{equation}
where $0<\beta<1$ and $\theta < 1$ may be chosen.
\end{proof}
Notice the choice of $\delta$ small enough to satisfy $\alpha^2 < 1$ is always possible, as each term may be independently driven below unity by a sufficiently small value of $\delta$, so long as the quasi-orthogonality constant $\Lambda$ is sufficiently close to one.  For a discussion on the optimal contraction factor see Remark 4.3 in~\cite{CKNS08}; see also the discussion in~\cite{HTZ09a}.

\subsection{Convergence of GOAFEM}
   \label{subsec:goalConvergenc1e}

We now derive a bound on error in the goal function.
\begin{theorem}[GOAFEM functional convergence]
Let the problem data satisfy Assumption~\ref{data_assumption1s} and the mesh satisfy Assumption~\ref{mesh_assumption1s}.  Let $u$ the solution to~\eqref{primal_proble1m} and $z$ the solution to~\eqref{dual_proble1m}.
Let $\theta \in (0,1]$, and let $\{\cT_k, \V_k, u_k, z_k\}_{k \ge 0}$ be the sequence of meshes, finite element spaces and discrete primal and dual solutions produced by GOAFEM.  Let $\gamma_p$ the constant $\gamma$ from Theorem~\ref{contraction_eac1h} applied to the primal problem~\eqref{discrete_prima1l} and $\gamma_d$ the constant $\gamma$ from Theorem~\ref{contraction_eac1h} applied to the dual~\eqref{discrete_dua1l}.  Then for constant $\alpha < 1$ as determined by Theorem~\ref{contraction_eac1h}
\end{theorem}
 \begin{align*}\label{goal_contraction_res}
  |g(u) - g(u_k)| & \le 2\left\{ \alpha ^{2k} \left(  \nrse{u - u_0}^2 
 +   \gamma_p \eta_0^2(u_0, \cT_0)  \right) - \gamma_p\eta_k^2 \right\}^{1/2} \nonumber \\
& \quad \times  \left\{ \alpha^{2k} \left(  \nrse{z - z_0}^2 
+  \gamma_d \zeta_0^2(z_0, \cT_0) \right)  - \gamma_d \zeta_k^2\right\}^{1/2}.
 \end{align*}

\begin{proof}
On the primal side for all $v_k \in \V_k$
\[
a(u - u_k, v_k) = a(u, v_k) - a(u_k, v_k) = f(v_k) - f(v_k) = 0,
\]
the primal Galerkin orthogonality property.
On the dual side, $g(u) = a^\ast(z, u)$ and $g(u_k) = a^\ast(z,u_k, )$ so that
\begin{align}\label{goal_to_a}
g(u) - g(u_k)  = a^\ast(z,u - u_k)
 = a(u - u_k, z)
 = a(u - u_k, z - z_k).
\end{align}
Define an inner-product $\alpha$ by the symmetric part of $\forma$
\[
\alpha(v,w) = \langle A \grad v, \grad w \rangle + \langle c  v, w \rangle,
\]
then
\[
\nrse v^2 = a(v,v) = \alpha(v,v),
\]
and
\[
a(v,w) = \alpha(v,w) + \langle b \cdot \grad v, w\rangle.
\]
Then as $\alpha(\cdot, \cdot)$ is a  symmetric bilinear form on  Hilbert space; it is an inner product and it induces a norm identical to the energy norm induced by $\forma$.  As such we may apply the Cauchy-Schwarz inequality~\cite{Evans98} to $\alpha$  and  we're left to handle the convection term.
\begin{align}\notag
a(u - u_k, z - z_k) &= \alpha(u - u_k, z-z_k) + \langle b \cdot \grad ( u - u_k), z - z_k \rangle \\
& \le \nrse{u - u_k}\nrse{z - z_k} +  \langle b \cdot \grad ( u - u_k), z - z_k \rangle.
\end{align}

By H\"older's inequality followed by a duality estimate as in~\S\ref{subsec:duality} on the dual error and coercivity on the primal,
\begin{align} 
\langle b \cdot \grad(u - u_k), z - z_k \rangle 
  \le    \nr b_{L_\infty} C_\ast h_0^{s} \mu_0^{-1/2} \nrse{z - z_k}\nrse{u - u_k}.
 \end{align}
Recalling $\bar \Lambda = \nr b_{L_\infty} C_\ast h_0^s \mu_0^{-\half}$
 \begin{align}
 a(u - u_k, z - z_k) & \le  \nrse{u - u_k}\nrse{z - z_k} +  \bar \Lambda   \nrse{u - u_k} \nrse{z - z_k}.
 \end{align}
Under assumption~\eqref{fine_init_mes1h} $(\bar \Lambda < 1)$ on the initial mesh and from (\ref{goal_to_a}), 
 \begin{align}\label{goal_bound}
 |g(u) - g(u_k)| = |a(u - u_k, z - z_k)| \le  2\nrse{u - u_k} \nrse{z - z_k} .
 \end{align}
 From~\ref{subsec:contractio1n} there is an $\alpha < 1$ such that for the primal problem with estimator $\eta_k$
 \begin{equation}\label{g1iterp}
 \nrse{u - u_{k+1}}^2 \le \alpha^2 \left( \nrse{u - u_k}^2 + \gamma_p \eta_{k}^2 \right) - \gamma_p \eta_{k+1}^2
 \end{equation}
 and for the dual problem with estimator $\zeta_k$
  \begin{equation}\label{g1iterd}
 \nrse{z - z_{k+1}}^2   \le \alpha^2 \left( \nrse{z - z_k}^2 + \gamma_d  \zeta_{k}^2 \right) - \gamma_d \zeta_{k+1}^2.
 \end{equation}
Iterating, we have from \eqref{g1iterp} and \eqref{g1iterd}
\begin{align}\label{g2iterp}
 \nrse{u - u_{k}}^2 + \gamma_p \eta_{k}^2&\le \alpha^{2k} \left( \nrse{u - u_0}^2 + \gamma_p \eta_{0}^2 \right)  \\
 \label{g2iterd}
 \nrse{z - z_{k}}^2 +\gamma_d \zeta_{k}^2 &\le \alpha^{2k} \left( \nrse{z - z_0}^2 + \gamma_d  \zeta_{0}^2 \right). 
\end{align}

From \eqref{goal_bound}, \eqref{g2iterp} and  \eqref{g2iterd} obtain the contraction of error in quantity of interest
 \begin{align}\label{goal_contraction}
  |g(u) - g(u_k)| & \le 2\left\{ \alpha ^{2k} \left(  \nrse{u - u_0}^2 
 +   \gamma_p \eta_0^2(u_0, \cT_0)  \right) - \gamma_p\eta_k^2 \right\}^{1/2} \nonumber \\
& \quad \times  \left\{ \alpha^{2k} \left(  \nrse{z - z_0}^2 
+  \gamma_d \zeta_0^2(z_0, \cT_0) \right)  - \gamma_d \zeta_k^2\right\}^{1/2},
 \end{align}

or more simply
 \begin{align}\label{goal_contraction1}
  |g(u) - g(u_k)| + \gamma_p \eta_k^2 + \gamma_d \zeta_k^2& \le  \alpha ^{2k} \left(  \nrse{u - u_0}^2 
 +   \gamma_p \eta_0^2(u_0, \cT_0) \right. \nonumber \\
& \qquad \left. + \nrse{z - z_0}^2 
+  \gamma_d \zeta_0^2(z_0, \cT_0) \right) \\
& = \alpha^{2k}Q_0^2,
 \end{align}
 with $Q_0$ the quasi-error on the initial mesh.

\end{proof}

\section{Complexity}
   \label{sec:complexity}

Following the discussion in~\cite{CKNS08}, we can bound the growth of the mesh by
\begin{align}
\# \cT_k - \# \cT_0& \le C \left\{
  Q_k(u_k, \cT_k)^{-1/s}  + Q_k(z_k, \cT_k)^{-1/t}  \right\} \nonumber \\
& \le
C \left\{  \left( \nrse{u - u_k}^2 + \gamma_p \osc_k^2 (u_k, \cT_k) \right)^{-1/2s}  \right.\nonumber\\ 
 & \quad + \left. \left( \nrse{z - z_k}^2 + \gamma_d \osc_k^2 (z_k, \cT_k) \right)^{-1/2t}   \right\}.
 \label{CSL1}
\end{align}
Here we make the usual approximation class assumptions on primal and dual solutions  $u \in \cA_s \an z \in \cA_t$. The second inequality in~\eqref{CSL1} follows from~\eqref{total_equiv_quasi} the equivalence of the total error and quasi-error.  As compared with the analogous result for the standard adaptive method in~\cite{CKNS08}, 
$\# \cT_k - \# \cT_0 \le C  \left( \nrse{u - u_k}^2 + \gamma_p \osc_k^2 (u_k, \cT_k) \right)^{-1/2s} $, the same procedure applied to the goal-oriented method bounds the complexity but fails to produce an optimal bound as is shown for the Laplacian in~\cite{MoSt09}.  We find in our numerical results, however, that our marking and refinement strategy is comparable to that presented in~\cite{MoSt09}.

To show the equivalence of the total error and quasi-error, we start with a fairly standard result that  may be found 
in~\cite{MeNo05} Lemma 3.1 and a similar result in~\cite{MoSt09} Proposition 4.3 and Corollary 4.4. 

\begin{lemma}[Global lower bound]\label{lemma_global_lower}
Let the problem data satisfy Assumption~\ref{data_assumption1s} and the mesh satisfy  Assumption~\ref{mesh_assumption1s}. 
Let $\cT_1, \cT_2 \in \T$ and $\cT_2 \ge \cT_1$ a full refinement. Let $u_k \in \V_k$ the solution to~\eqref{discrete_prima1l}, $k = 1,2$.  Then there is a global constant $c_2 > 0$ such that

\begin{equation}\label{lowerbd}
c_2 \eta_1^2(u_1, \cT_1) \le \nrse{u - u_1}^2 + \tosc_1^2(u_1, \cT_1).
\end{equation} 
\end{lemma}

Putting together  the lower bound on total error~\eqref{lowerbd} with domination of the error estimator over the oscillation and  we have the equivalence of the total error and quasi-error
\begin{align}\notag
\nrse{u - u_j}^2 + \gamma_p \osc_j^2 (u_j, \cT_j) & \le \nrse{u - u_j}^2 + \gamma_p \eta_j^2(u_j, \cT_j) \\
\label{total_equiv_quasi}
& \le \left(  1 + \f {\gamma_p} {c_2}\right) E_j^2(u_j, \cT_j).
\end{align}
or
\begin{equation}\label{Qres2}
 E_j^{-1/s}(u_j, \cT_j) \le \left(  1 + \f {\gamma_p} {c_2}\right)^{1/2s}  Q_j^{-1/s}(u_j, \cT_j)
\end{equation}
and similarly for the dual problem 
\begin{equation}\label{dualQres2}
 E_j^{-1/t}(z_j, \cT_j) \le \left(  1 + \f {\gamma_d} {c_2}\right)^{1/2t}  Q_j^{-1/t}(z_j, \cT_j).
\end{equation}

The complexity bound relies on the approximation class assumption $u \in \cA_s$.  We refer the reader to the discussion of the class $\cA_s$ and related approximation classes in~\cite{CKNS08}.
These results imply optimality of the method in the sense of global error convergence, but not optimality in the sense of convergence to the goal functional.
However, our numerical results below appear to demonstrate this behavior with respect to goal convergence as well.

\section{Numerics}
\label{sec:numerics}

We demonstrate the performance of our method (HP) by comparing it with the two most relevant methods: the dual weighted residual method, DWR;  and the method in~\cite{MoSt09}, referred to here as MS which uses the same residual-based indicator as HP but with a different strategy for adaptive marking.  Our results indicate that on a wide variety of convection dominated problems, our method performs as well or in some cases better than the other two methods.  We also show a number of the adaptive meshes below, demonstrating that in many cases, the compared methods produce similar performance from qualitatively different adaptive refinements.  A discussion of the DWR method may be found in~\cite{BaRa03,Becker.R;Rannacher.R1996a,EHL02,Giles.M;Suli.E2003,Gratsch.T;Bathe.K2005,EHM01} for example.  In our DWR implementation, the finite element space for the primal problem 
$\V_{\cT_k}$ employs linear Lagrange elements as do HP and MS for both the primal and dual spaces.  For DWR, the dual finite element space $\V_{\cT_k}^2$ uses quadratic Lagrange elements. The DWR indicator estimates the influence of the dual solution on the primal residual. Elementwise
\begin{equation*}
\eta_{\cT_k} (v,T) \coloneqq \langle R(v), z^2 - I_k z^2\rangle_T  + \f 1 2  \langle J_T(v), z^2 - I_k z^2 \rangle_{\pa T},   \quad v \in \V_{\cT_k},
\end{equation*}
where $z^2 \in \V_{\cT_k}^2$ is the dual solution and $I_k$ is the interpolator onto $\V_{\cT_k}$. The error estimator is  the absolute value of the sum of indicators  
\begin{align*}
\eta_k =  \left| \sum_{T \in \cT_k} \eta_T(u_k,T) \right| \le\sum_{T \in \cT_k}  \left| \eta_T(u_k,T)\right|.  
\end{align*}
Both HP and MS use the residual based indicators described in this paper.

In the adaptive algorithms, we use the D\"orfler marking strategy with  $\theta = 0.6$ for all the methods. Our numerical experiments are implemented 
using FETK~\cite{Hols2001a},
which is a fairly standard set of finite element modeling libraries for 
approximating the solutions to systems of linear and nonlinear elliptic and parabolic 
equations.

We consider the convection dominated problem
\[
a(u,v) \coloneqq \f 1 {1000}\langle \grad u , \grad v\rangle + \langle b\cdot \grad u, v \rangle = f(v),
\]
with $b = (y, \f 1 2 - x)^T$. The goal function is $g(u) = \int g(x,y) u$ and 
\[
{g(x,y) = 500\exp(-500((x-x_1)^2 + (y - y_1)^2))}, \quad \text{ with }~(x_1,y_1) = (0.9,0.675). 
\]
The load function $f$ is chosen so that the exact solution $u$ is of the form
\[
u = \sin(\pi x) \sin (\pi y) ( 2((x - x_0)^2 +  (y - y_0))^2 + 10^{-3})^{-1}.
\]
We show results the the domain $\Omega = (0,1)^2 \setminus(1/3, 2/3)^2$ for four sets of parameters
\begin{align}
Figure~\ref{fig:HSD_311}  && x_0 = 0.1 &&  y_0 = 0.1 \label{HSD_311}, \\
				        && x_0 = 0.7 &&  y_0 = 0.1 \label{HSD_371}. \\
Figure~\ref{fig:HSD_317}  && x_0 = 0.1 &&  y_0 = 0.7 \label{HSD_317}. \\
Figure~\ref{fig:HSD_377}  && x_0 = 0.7 &&  y_0 = 0.7 \label{HSD_377}.
\end{align}
In these four examples, we see the effect of moving the primal data far from both the dual solution and the spike in the dual data, then colliding with the dual solution but remaining far from the dual spike and finally in close vicinity to the spike in the goal data $g(x,y)$. We start with an initial mesh of 128 elements. The numbers of iterations shown in each problem are selected to compare the error in each method with similar numbers of elements in each mesh.  HP and MS use linear Lagrange elements for both primal and dual finite element spaces whereas our implementation of DWR uses quadratic Lagrange elements for the dual finite element space and linears for the primal.   
\begin{figure}
\includegraphics[width=0.45\textwidth]{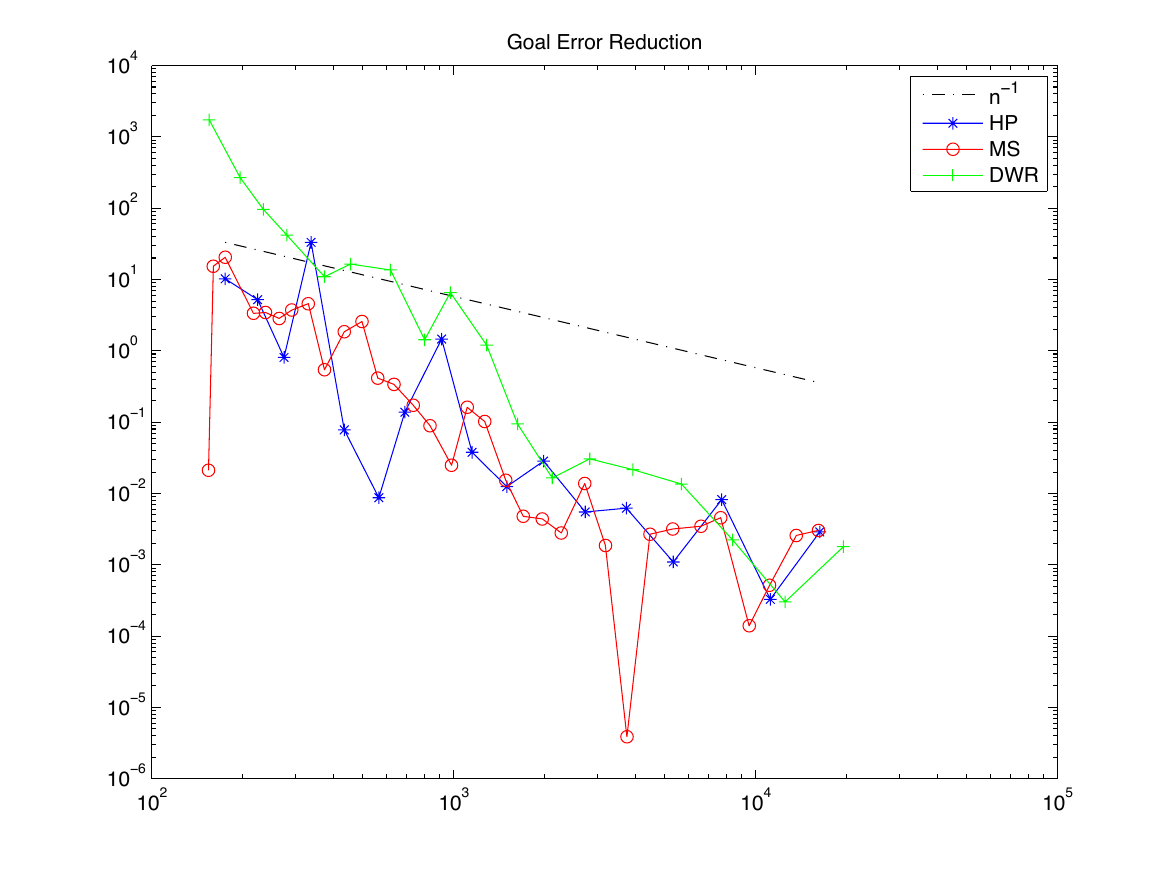}~
\includegraphics[width=0.45\textwidth]{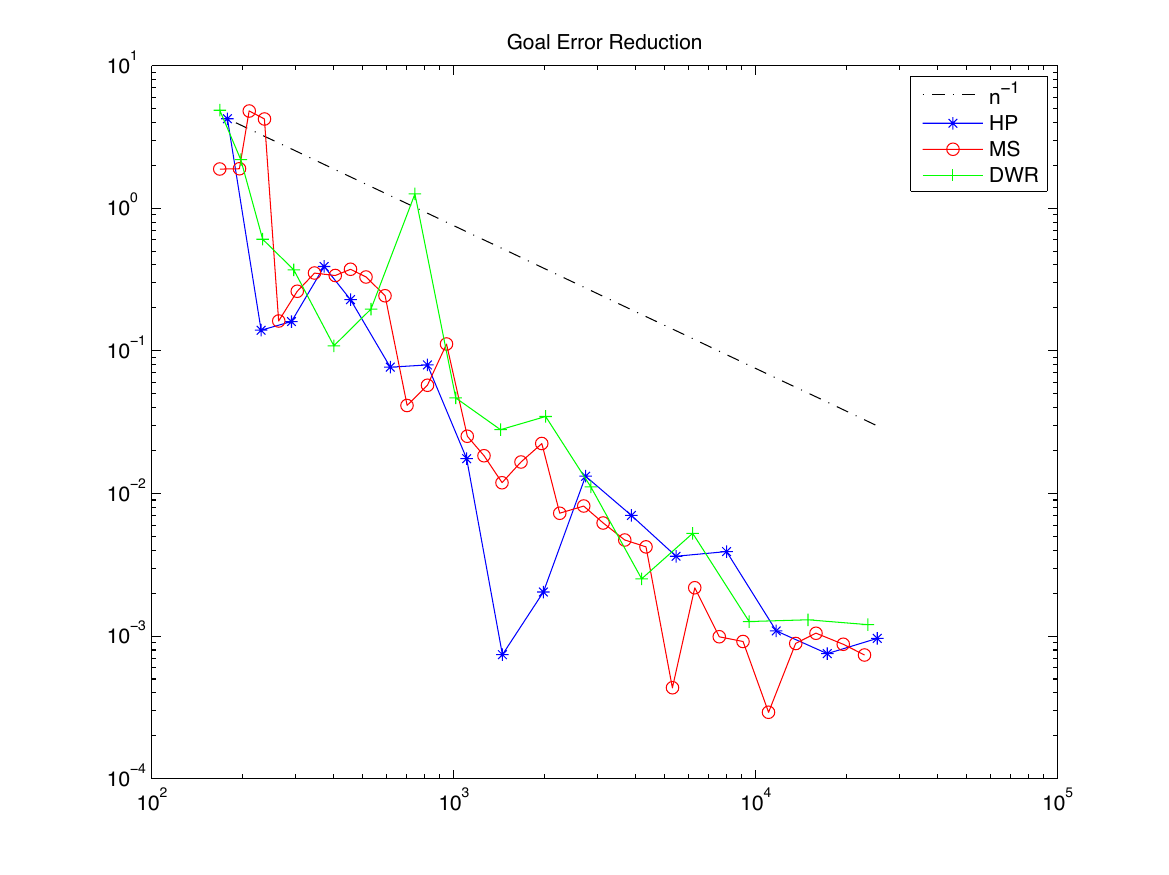}
\caption{Left: goal error after 18 HP, 36 MS and 19 DWR iterations for problem~\ref{HSD_311}, compared with $n^{-1}$.  Right: goal error after 18 HP, 36 MS and 17 DWR iterations for problem~\ref{HSD_371}, compared with $n^{-1}$. }
\label{fig:HSD_311}
\end{figure}
\begin{figure}
\includegraphics[width=0.3\textwidth]{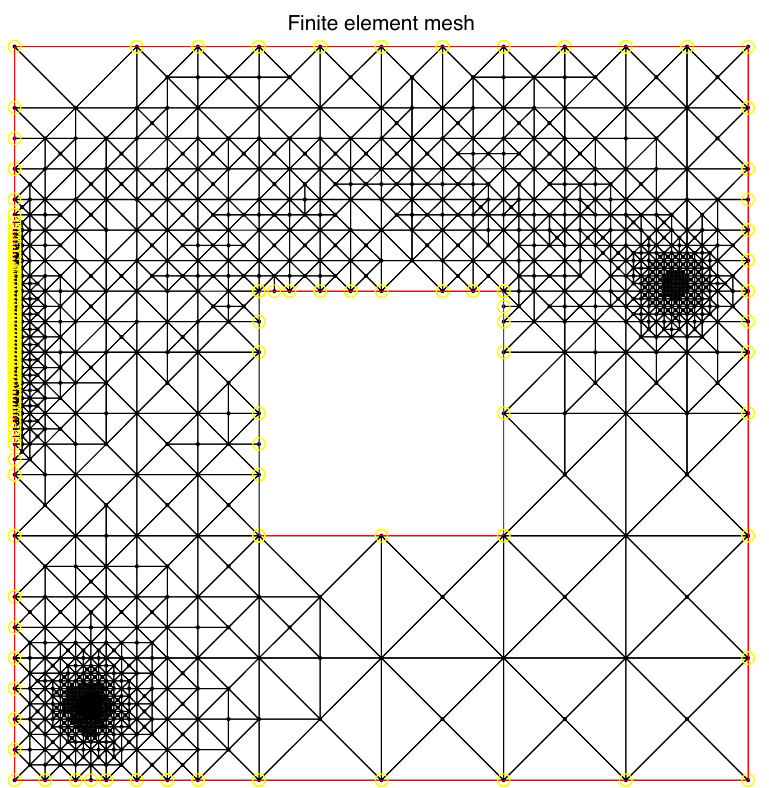}~
\includegraphics[width=0.3\textwidth]{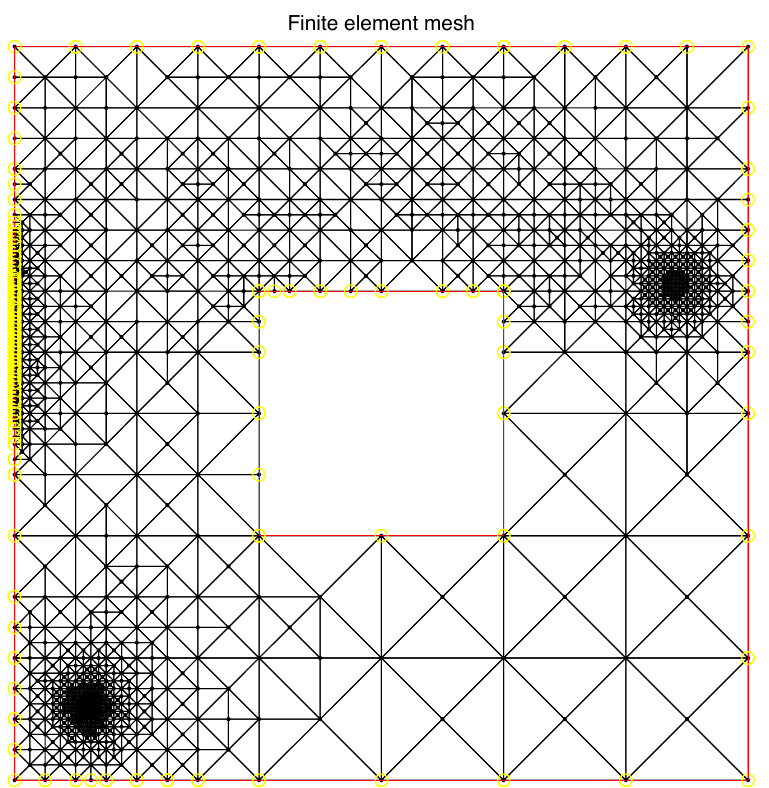}~
\includegraphics[width=0.3\textwidth]{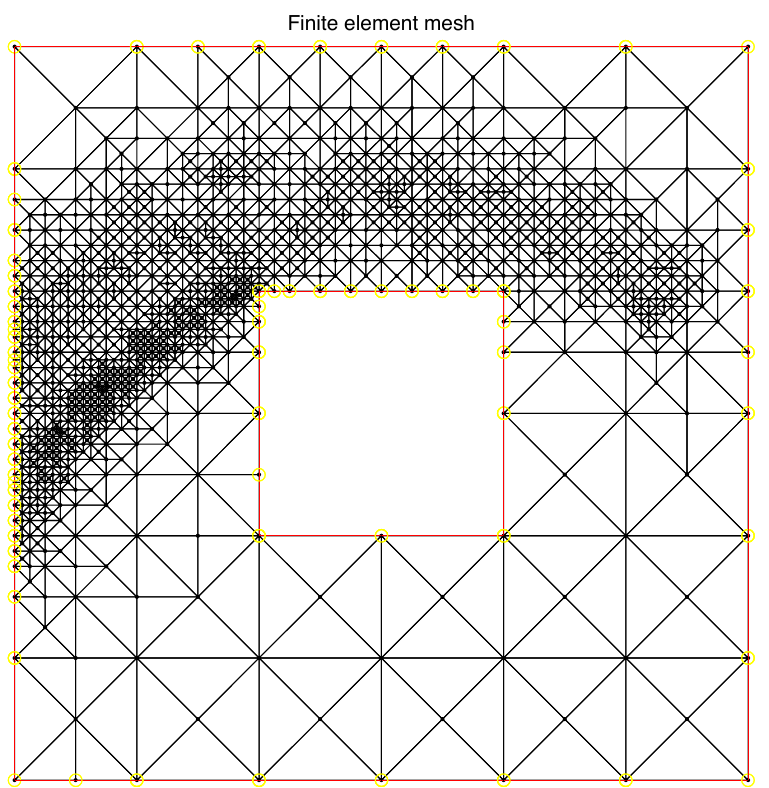}~
\caption{Left: 13 iterations of HP (2695 elements). Center: 26 iterations of MS (2715) elements). Right: 14 iterations of DWR (3045 elements) for problem~\ref{HSD_311}.  }
\label{fig:mesh_311}
\end{figure}

In the first two problems~\ref{HSD_311} and~\ref{HSD_371}, where the primal spike is far from both the dual spike and the dual solution, we see all three methods produce comparable reduction in the goal error shown in Figure~\ref{fig:HSD_311}; however, the residual based and DWR methods adaptively refine the mesh in qualitatively different ways. Figure~\ref{fig:mesh_311} shows comparable stages of mesh refinement for problem~\ref{HSD_311}.  DWR focuses refinement on the dual solution with the highest concentration of refinement along the side of the dual solution closest to the primal spike.  HP and MS both concentrate refinement on the primal and dual spikes as well as the interaction between the dual solution and the boundary.  HP and MS produce similar mesh refinements, and have a similar rate of error reduction compared to the number of mesh elements; however, MS takes generally twice as many iterations to obtain the same goal error as does HP.  As discussed in~\cite{BET11}, the increase in the number of iterations is a practical disadvantage as the code takes longer to run.

\begin{figure}
\includegraphics[width=0.6\textwidth]{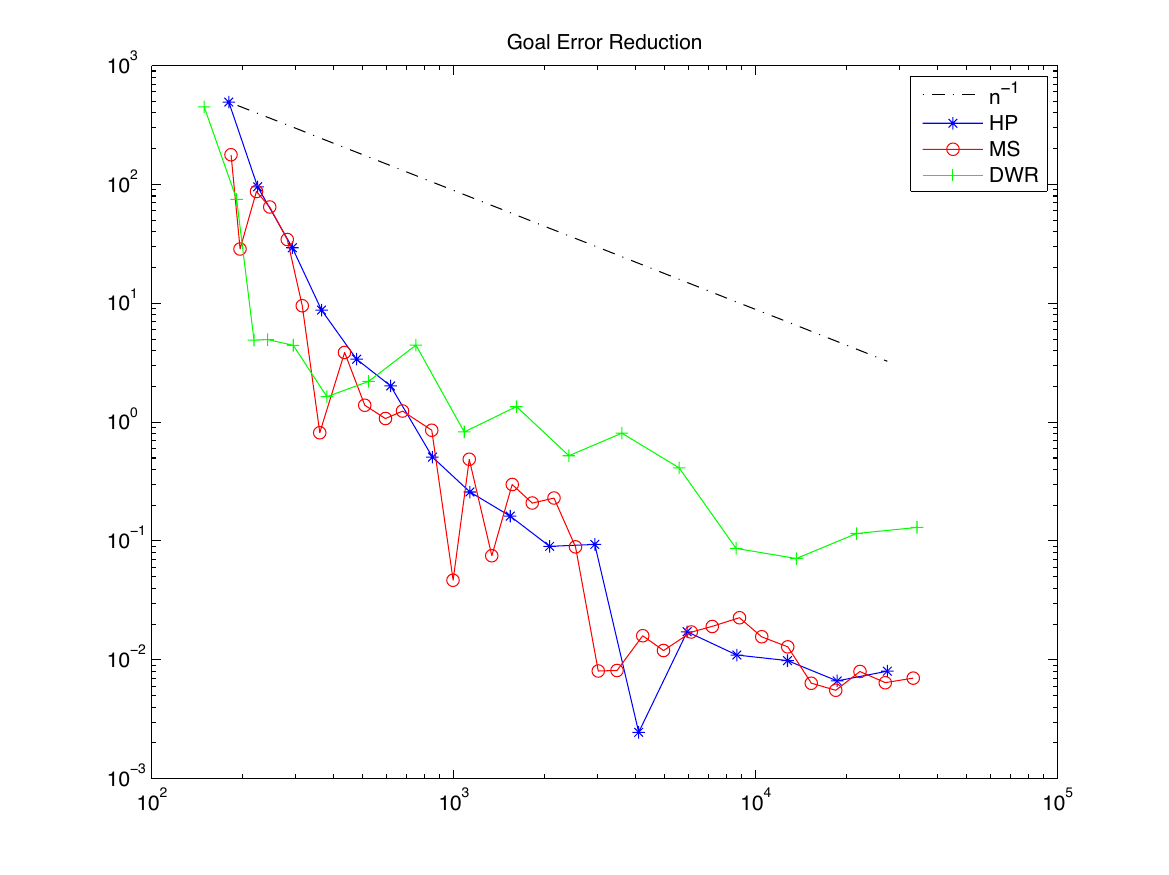}
\caption{Goal error after 18 HP, 36 MS and 18 DWR iterations for problem~\ref{HSD_317}, compared with $n^{-1}$.}
\label{fig:HSD_317}
\end{figure}
\begin{figure}
\includegraphics[width=0.4\textwidth]{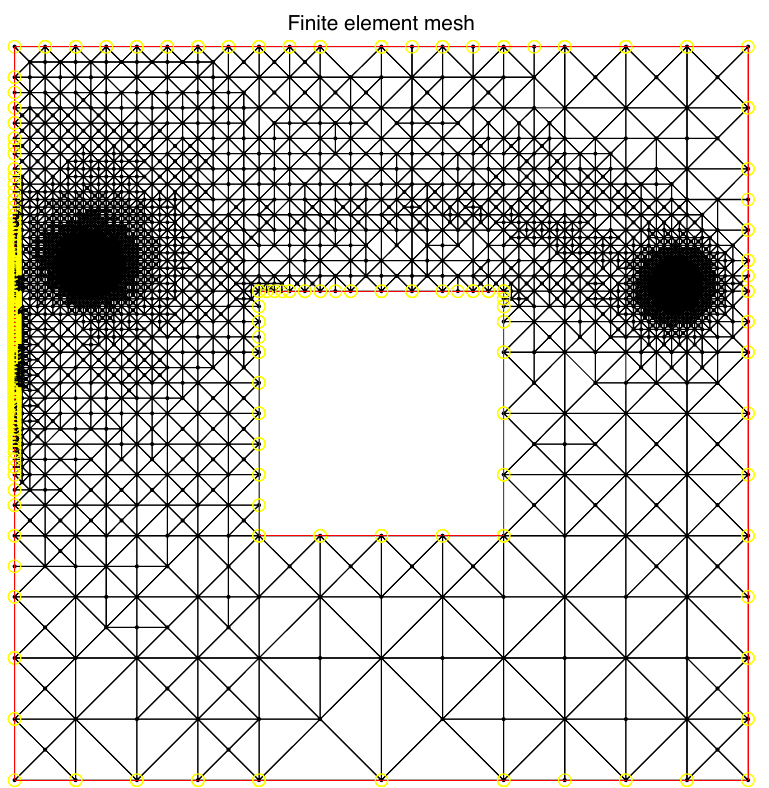}
\includegraphics[width=0.4\textwidth]{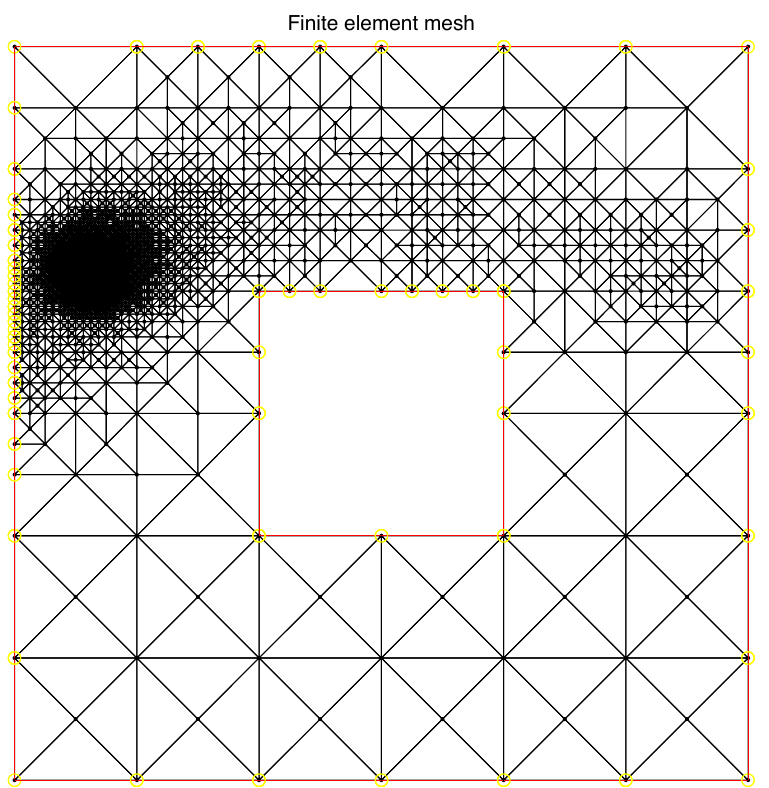}
\caption{Left: 15 iterations of HP (8667 elements). Right: 15 iterations of DWR (8621 elements) for problem~\ref{HSD_317}.  }
\label{fig:mesh_317}
\end{figure}

In the second problem~\ref{HSD_317} where the spikes in the primal and dual data are remote, but the dual solution collides with the primal data, the residual based methods both outperform DWR.  As seen in Figure~\ref{fig:mesh_317}, HP refines for both primal and dual data spikes as well as the boundary interacting with the dual solution.  DWR concentrates refinement on the interaction of the primal and dual solution which captures the spike in the primal data, but neglects  to refine either for the dual data spike, or the area surrounding the primal spike lying outside the path of the dual solution.  This problem is further investigated with respect to the strength of the diffusion term in problems~\ref{HSD_d1317}~-~\ref{HSD_d5317}, from which we see the area surrounding the primal spike is of greater importance.
\begin{figure}
\includegraphics[width=0.6\textwidth]{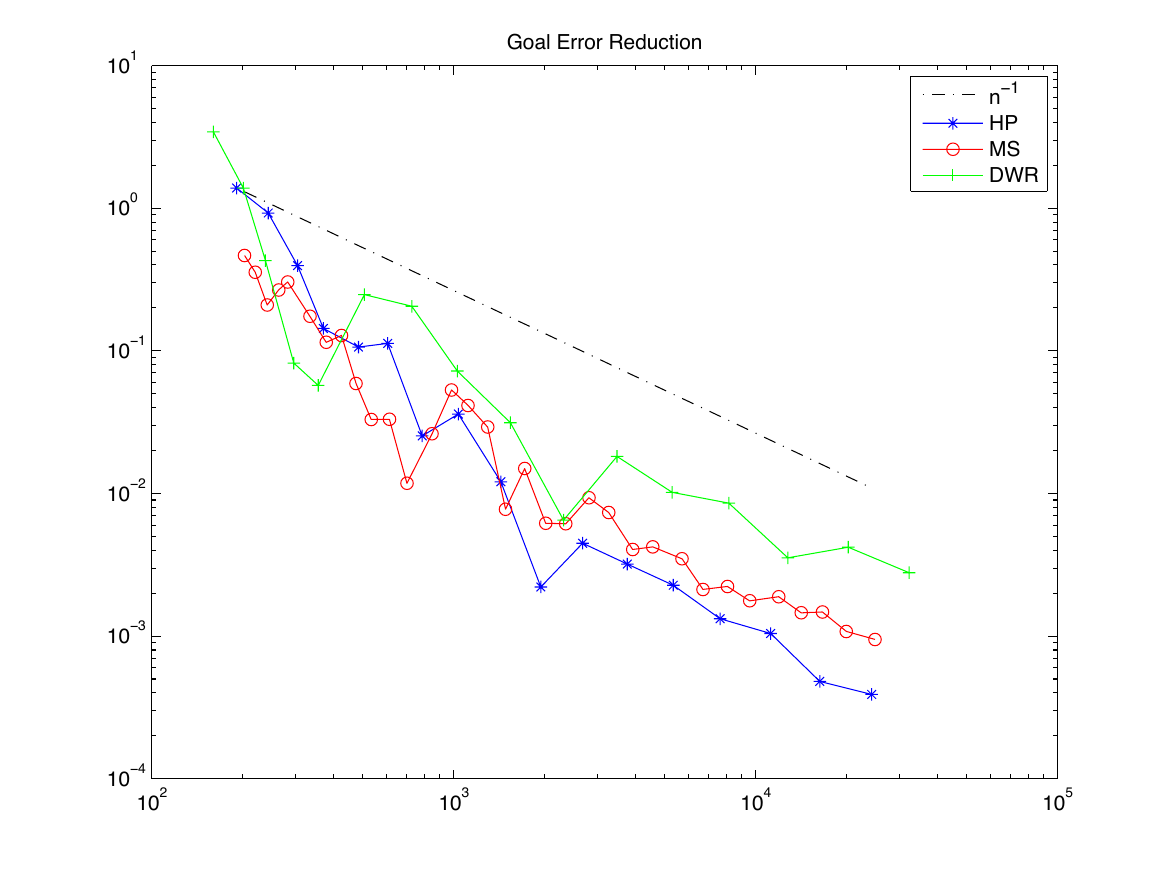}
\caption{Goal error after 18 HP, 36 MS and 17 DWR iterations for problem~\ref{HSD_377}, compared with $n^{-1}$.}
\label{fig:HSD_377}
\end{figure}
\begin{figure}
\includegraphics[width=0.4\textwidth]{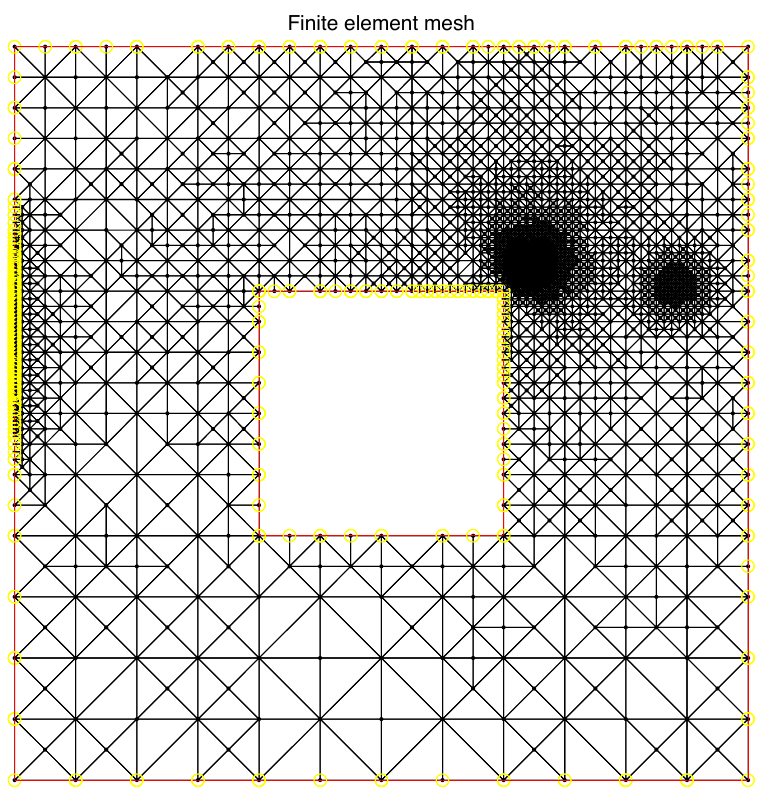}
\includegraphics[width=0.4\textwidth]{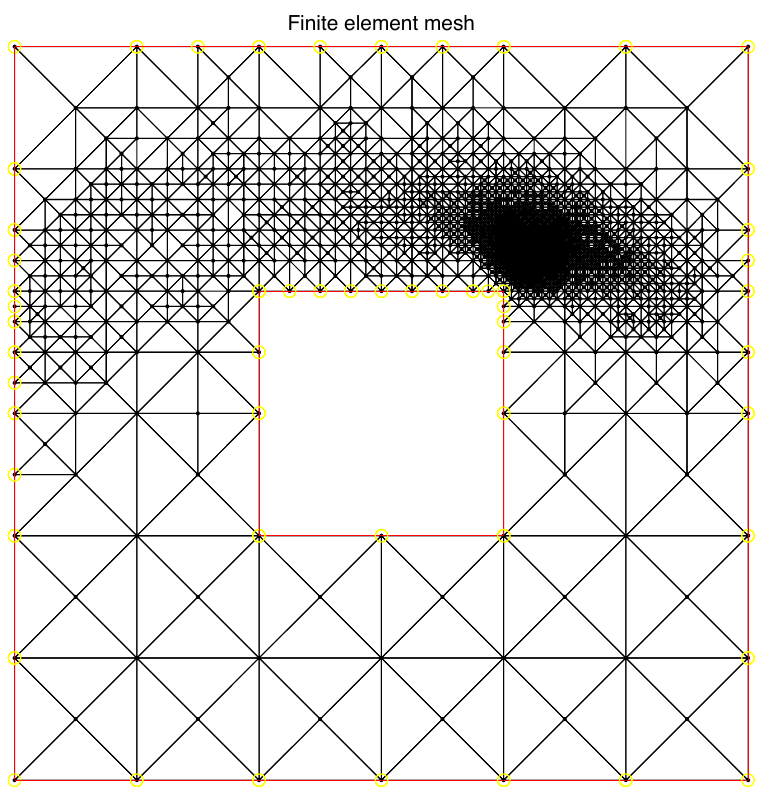}
\caption{Left:  14 iterations of HP (5332 elements). Right: 13 iterations of DWR (5292 elements) for problem~\ref{HSD_377}. }  
\label{fig:mesh_377}
\end{figure}

Problem~\ref{HSD_377} places the primal and dual spikes in close vicinity. As seen in Figures~\ref{fig:HSD_377} and~\ref{fig:mesh_377}, the residual based methods and HP in particular outperform the DWR method even when the primal and dual spikes are close.  Restricting refinement to the close vicinity of the dual solution, it appears in this convection problem the DWR indicator is missing important information about the structure of the primal data that the residual based methods do a better job of capturing in their adaptive refinement.

In the next four problems we change the scale of the diffusion term in problem~\ref{HSD_317} to investigate the relative performance of the three methods when dominance of the convection term is either increased or decreased.

\begin{align}
Figure~\ref{fig:HSD_d12317} && a(u,v) \coloneqq 10^{-1}\langle \grad u , \grad v\rangle 
+ \langle b\cdot \grad u, v \rangle  \label{HSD_d1317}, \\
 && a(u,v) \coloneqq 10^{-2}\langle \grad u , \grad v\rangle 
+ \langle b\cdot \grad u, v \rangle  \label{HSD_d2317}. \\
Figure~\ref{fig:HSD_d45317} && a(u,v) \coloneqq 10^{-4}\langle \grad u , \grad v\rangle 
+ \langle b\cdot \grad u, v \rangle \label{HSD_d4317}, \\
&& a(u,v) \coloneqq 10^{-5}\langle \grad u , \grad v\rangle 
+ \langle b\cdot \grad u, v \rangle \label{HSD_d5317}.
\end{align}

\begin{figure}
\includegraphics[width=0.45\textwidth]{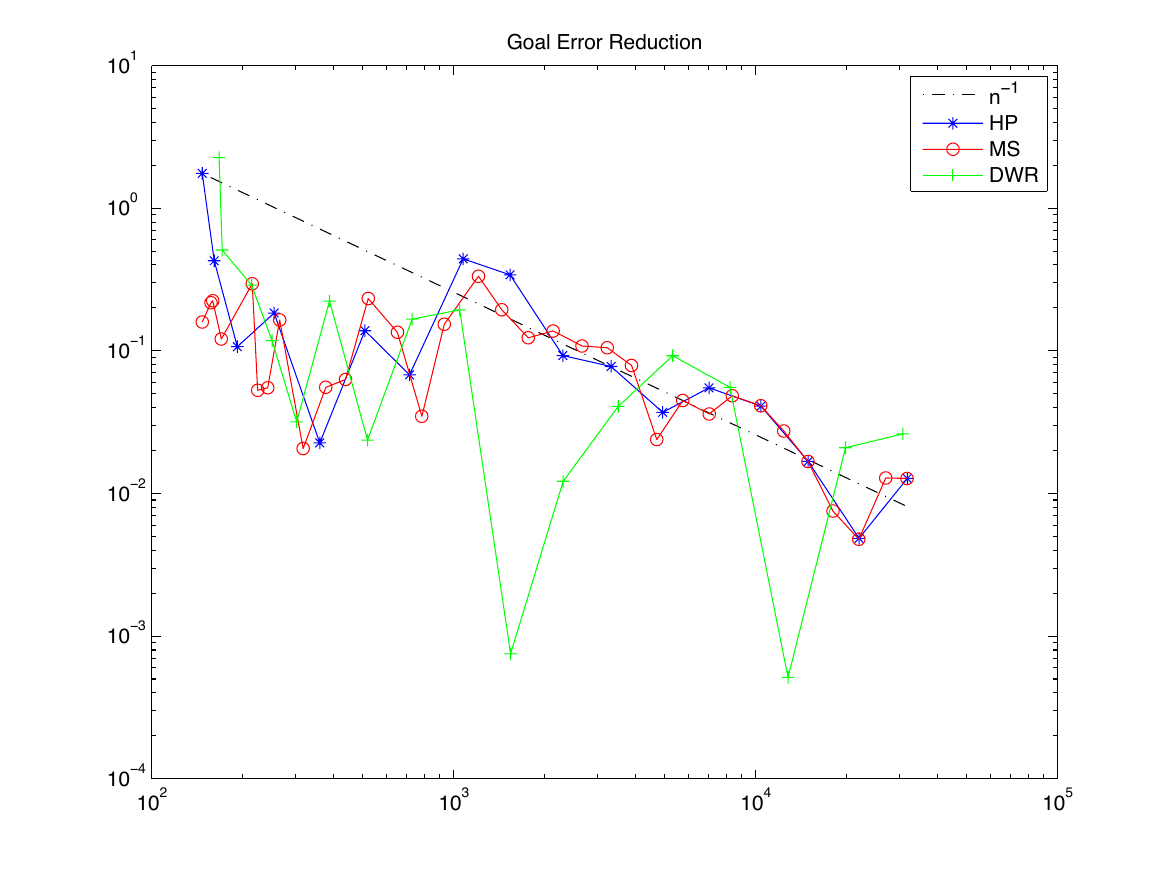}~
\includegraphics[width=0.45\textwidth]{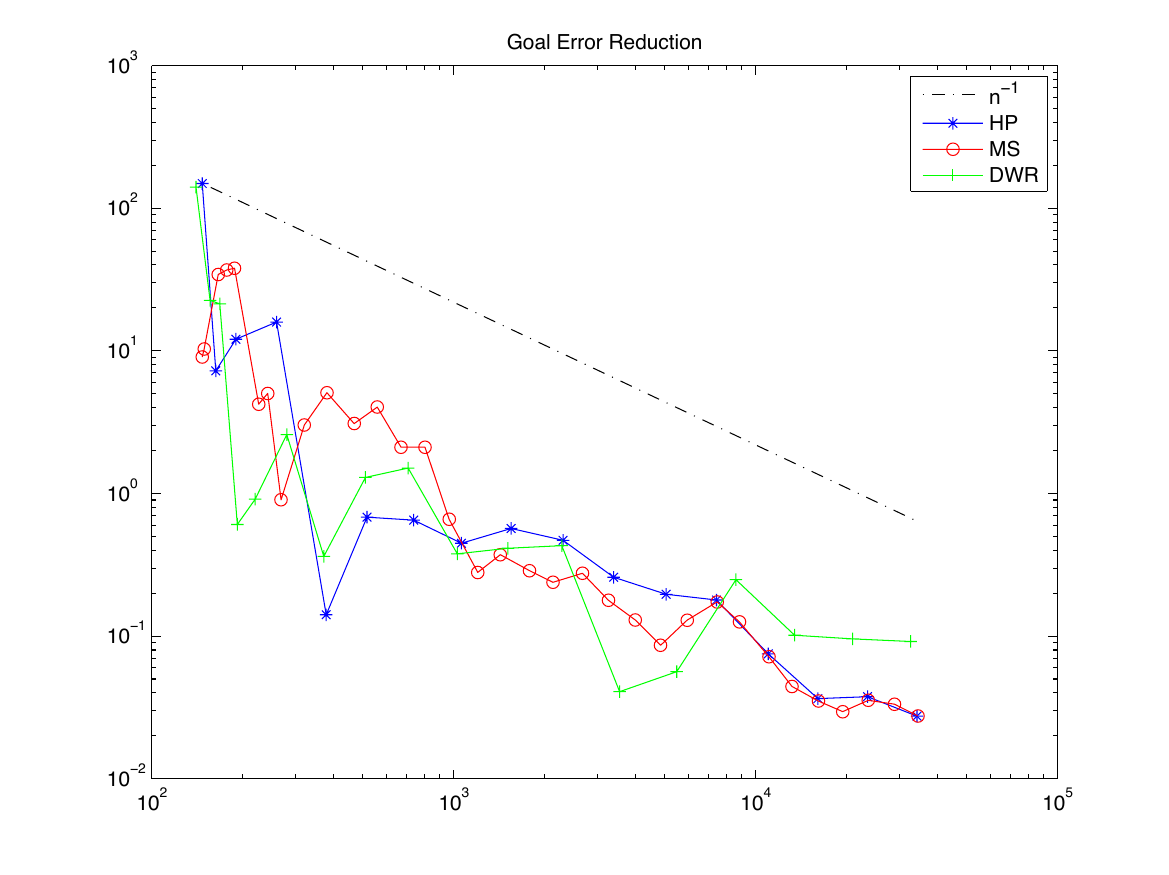}~
\caption{Left: Goal error after 18 HP, 36 MS iterations and18 DWR iterations for problem~\ref{HSD_d1317}, compared with $n^{-1}$. Right: Goal error after  18 HP, 36 MS and 19 DWR iterations compared with $n^{-1}$.}
\label{fig:HSD_d12317}
\end{figure}

Figure~\ref{fig:HSD_d12317} shows the reduction in goal error for problems~\ref{HSD_d1317} and ~\ref{HSD_d2317}, in which the dominance of the convection term has been decreased.   The graph on the left shows HP with a mesh of 31806 elements, MS with 31765 elements and DWR with 30748 elements.  
The graph on the right shows HP with a mesh of 34249 elements, MS with 34473 elements and DWR with 32580 elements.
We see as the diffusion coefficient in increased, the three methods show comparable performance, but as the diffusion coefficient is decreased as in problem~\ref{HSD_d2317}, the residual based methods start to perform better approaching the asymptotic regime.  

\begin{figure}
\includegraphics[width=0.45\textwidth]{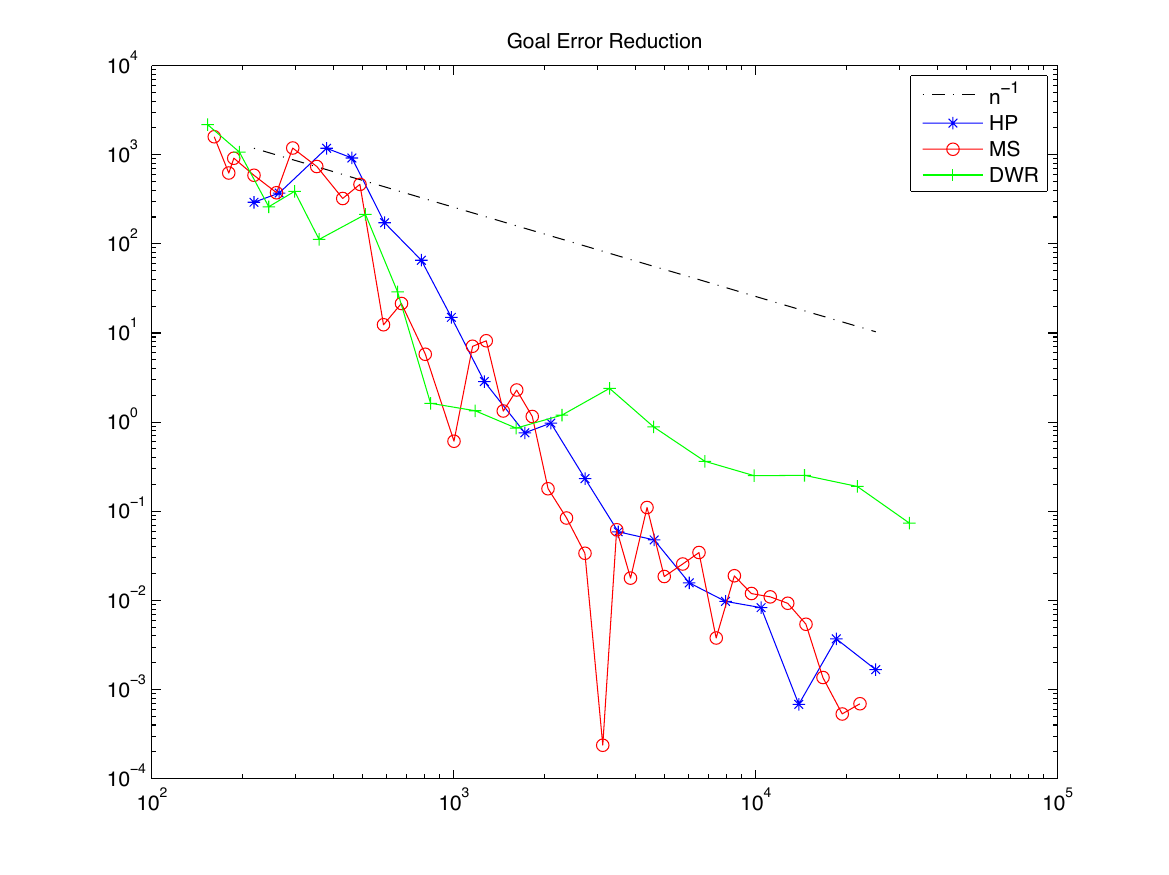}~
\includegraphics[width=0.45\textwidth]{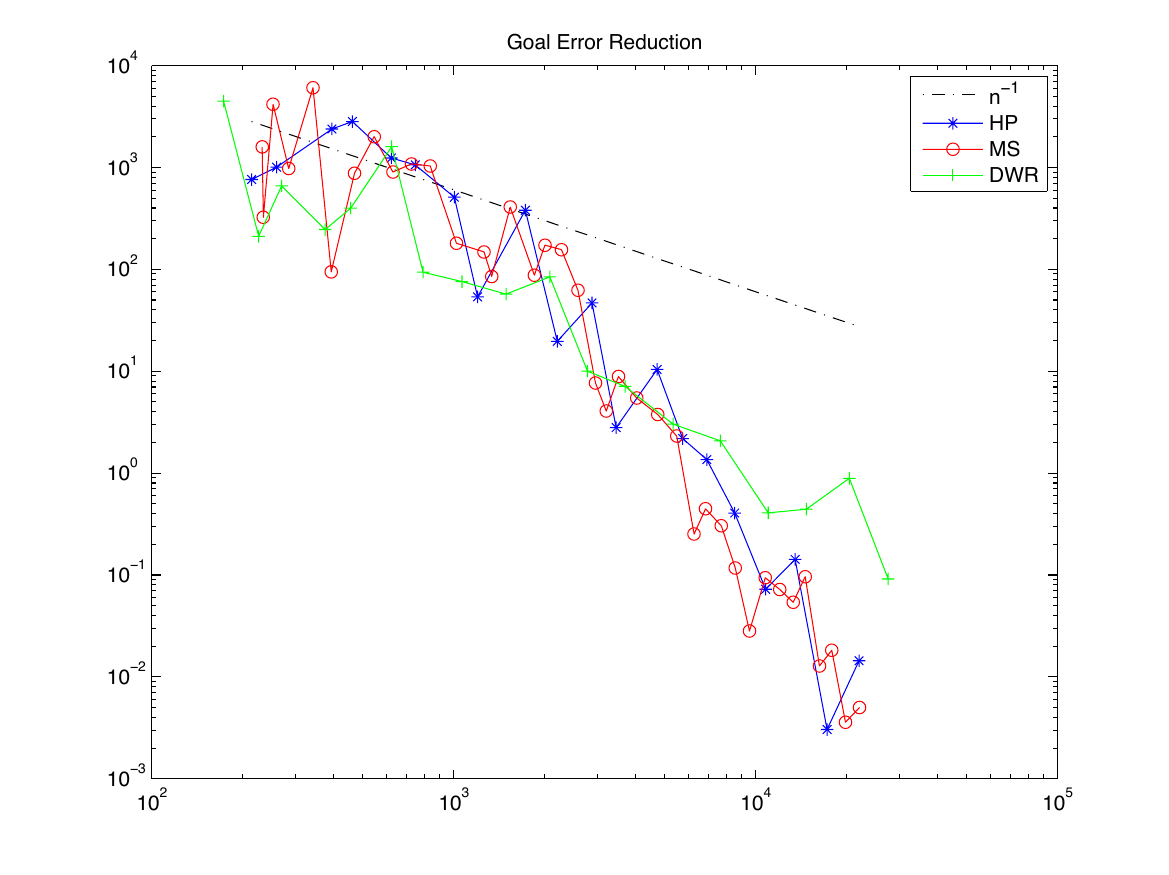}~
\caption{Left: Goal error after 20 HP, 40 MS iterations and19 DWR iterations for problem~\ref{HSD_d4317}, compared with $n^{-1}$. Right: Goal error after 21 HP, 41 MS and 19 DWR iterations for problem~\ref{HSD_d5317} compared with $n^{-1}$.}
\label{fig:HSD_d45317}
\end{figure}

\begin{figure}
\includegraphics[width=0.4\textwidth]{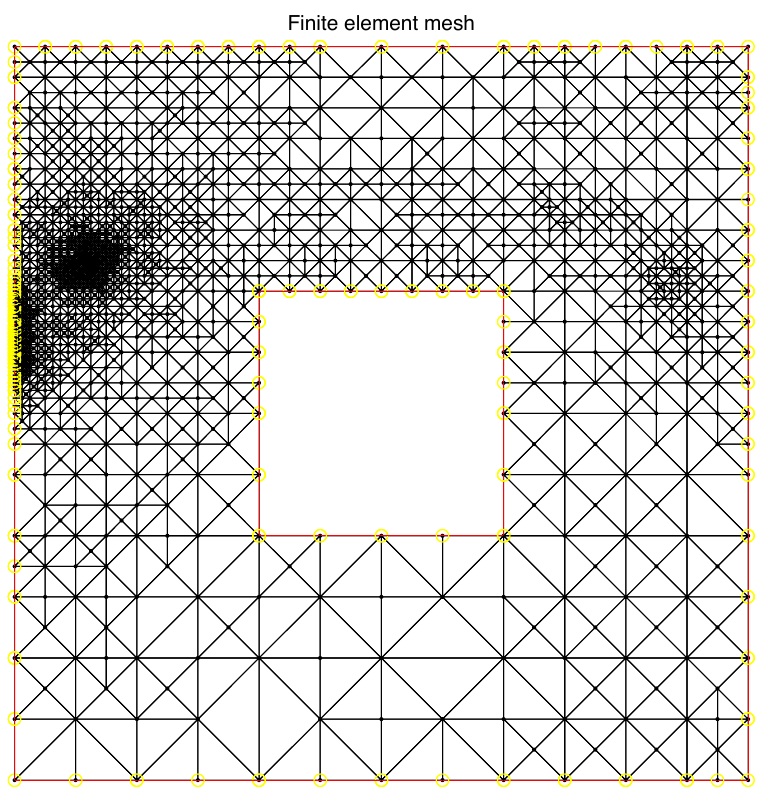}
\includegraphics[width=0.4\textwidth]{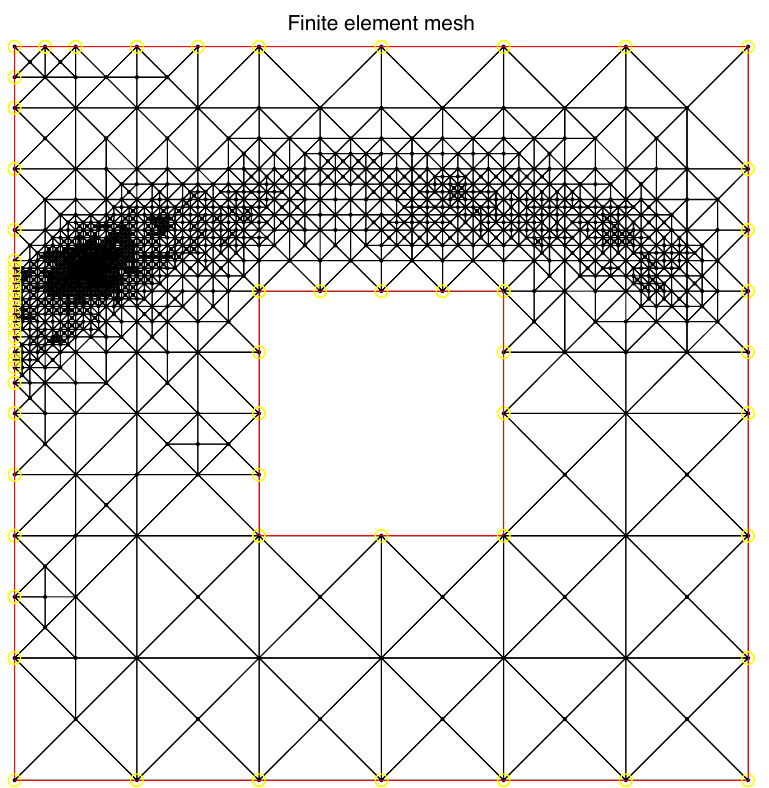}
\caption{Left:  13 iterations of HP (3507 elements). Right: 13 iterations of DWR (3281 elements) for problem~\ref{HSD_d4317}. }  
\label{fig:mesh_d4317}
\end{figure}

This trend is continued as the diffusion coefficient is decreased: all methods show a comparable performance for the first few iterations, then the residual based methods perform better than DWR as the number of mesh refinements increases.
Figure~\ref{fig:HSD_d45317} shows the reduction in goal error for problems~\ref{HSD_d4317} and ~\ref{HSD_d5317}.   The graph on the left shows HP with a mesh of 24988 elements, MS with 22150 elements and DWR with 32315 elements.
The graph on the right shows HP with a mesh of 21991 elements, MS with 22087 elements and DWR with 27493 elements.  As the diffusion term is decreased, the refinement for the DWR method is confined to a tighter band following the dual solution across the domain, allowing for insufficient development in the structure of the primal problem, as in Figure~\ref{fig:mesh_d4317}.  Unlike problem~\ref{HSD_317} with its lesser degree of convection dominance, the HP mesh for problem~\ref{HSD_d4317} shows a greater degree of refinement in the area surrounding the primal spike which is apparently important information for an accurate approximation of $g(u)$.

\section{Conclusion}
\label{sec:conc}

In this article we developed convergence theory for a class of 
goal-oriented adaptive finite element methods for 
second order nonsymmetric linear elliptic equations.
In particular, we established contraction and quasi-optimality results 
for a method of this type for the elliptic problem~\eqref{eqn:strong-eqn1}--\eqref{eqn:strong-b1c}
with $A$ Lipschitz, almost-everywhere symmetric positive definite (SPD),
with $b$ divergence-free, and with $c \ge 0$.
We first described the problem class in some detail,
with a brief review of conforming finite element 
discretization and error-estimate-driven adaptive finite element methods (AFEM).
We then described a goal-oriented variation of standard AFEM (GOAFEM).
Following the recent work of Mommer and Stevenson~\cite{MoSt09} for 
symmetric problems, we established contraction of GOAFEM.
We also showed convergence in the sense of the goal function.
Our analysis approach was signficantly different from that of Mommer and
Stevenson~\cite{MoSt09}, and involved the combination of
the recent contraction frameworks of Cascon, Kreuzer, Nochetto and Siebert ~\cite{CKNS08},
Nochetto, Siebert and Veeser~\cite{NSV09},
and of Holst, Tsogtgerel and Zhu~\cite{HTZ09a}.
Our numerical experiments demonstrate that our choice of marking strategy, while different from the choice used in~\cite{MoSt09} to show optimal complexity for the Laplacian, performs as well as that method on convection-diffusion problems over a wide range of convection dominance.  Our comparison to the standard goal-oriented strategy DWR shows that even for linear problems, the residual based indicators can be as effective and in some cases even outperform the dual weighted residual method when the D\"orfler marking strategy with the same parameter is used for all methods.  We emphasize that our comparison is based on the error in the goal function vs. the number of elements in the mesh. The DWR method is implemented with quadratic basis functions in the dual finite element space; however, the increase in degrees of freedom for that method is not indicated in our plots, nor has it in these examples given  the DWR method an advantage over the residual based methods.

Problems that were not yet addressed include allowing for
jump discontinuities in the diffusion cofficient, and allowing
for lower-order nonlinear terms.
We will address these aspects in a future work.

\section{Appendix}
\label{subsec:duality}
{\bf Duality.}
We include an appendix discussion of the duality argument used in the quasi-orthogonality estimate in an effort to make the paper more self-contained.

Let $u$ the variational solution to~\eqref{primal_proble1m} and $u_1 \in \V_1$ the Galerkin solution to~\eqref{discrete_prima1l}. Assume for any $g \in L_2(\Omega)$ the solution $w$ to the dual problem~\eqref{dual_proble1m} belongs to $H^2(\Omega) \cap H_0^1(\Omega)$ and
\begin{equation}\label{elliptic_reg0A}
|w|_{H^2(\Omega)} \le K_R \nr g_{L_2(\Omega)}.
\end{equation}
Then
\begin{equation}\label{duality_res0A}
\nr{u - u_1}_{L_2} \le C h_0 \nrse{u - u_1}.
\end{equation}

If $w \in H^{2}_{\text{loc}}(\Omega) \cap H_0^1(\Omega)$ but $w \notin H^2(\Omega)$ due to the angles of a nonconvex polyhedral domain $\Omega$ then $w \in H^{1+s}$ for some $0 < s < 1$ where $s$ depends on the angles of $\pa \Omega$.   Assume in this case for any $g \in L_2$
\begin{equation}\label{elliptic_reg1A}
|w|_{H^{1+s}(\Omega)} \le K_R \nr g_{L_2(\Omega)}
\end{equation}
then
\begin{equation}\label{duality_re1A}
\nr{u - u_1}_{L_2} \le C h_0^s \nrse{u - u_1}.
\end{equation}

As discussed in~\cite{Ci02},~\cite{Evans98} and~\cite{AOB01} the regularity assumptions are reasonable based on the continuity of the diffusion coefficients $a_{ij}$ and the convection and reaction coefficients $b_i$ and $c$ in $L_\infty(\Omega)$.

\iic{Proof of~\eqref{duality_res0A}:}  The proof follows the duality arguments in~\cite{AOB01} and~\cite{BS08}.

Let $w \in H_0^1(\Omega)$ the solution to the dual problem
\begin{equation}\label{duality_dual}
a^\ast(w,v) = \langle u - u_1, v \rangle, \quad v \in H_0^1(\Omega).
\end{equation}

Let $\cI^h$ a global interpolator based on refinement $\cT_1$.  Assume $\cI^h w$ is  $C^0$ and the corresponding shape functions have approximation order $m$.  For $m = 2$
\begin{equation}\label{interpolation_estA}
\nr{w - \cI^h w}_{H^1} \le C_\cI h_{\cT_1} |w|_{H^2.}
\end{equation}

As discussed in~\cite{AOB01} the interpolation estimate over reference element $\hat T$ follows from  the Bramble-Hilbert lemma applied to the bounded linear functional $f(\hat u) = \langle \hat u - \cI^h \hat u, \hat v \rangle$ where $\hat v \in H^t(\hat T)$ is arbitrary then set to $\hat u - \cI^h \hat u$.  The  Sobolev semi-norms for $t = 0,1$ over elements $T \in \cT$ are bounded via change of variables to the reference element.  Summing over $T \in \cT$ and combining semi-norms into a norm estimate establishes ~\eqref{interpolation_estA}.

By~\eqref{elliptic_reg0A} we have the bound
\begin{equation}\label{elliptic_regA}
|w|_{H^2} \le K_R \nr {u - u_1}_{L_2}.
\end{equation}

By the identity $a(v,y) = a^\ast(y,v)$ write the primal form of the variational problems
\begin{align}
\label{primal1A}
a(u,v) & = f(v), \quad v \in H_0^1(\Omega) \\
\label{primal2A}
a(u_1,v) & = f(v), \quad v \in \V_1 \\
\label{primal3A}
a(v,w) & =  \langle u - u_1, v \rangle, \quad v \in H_0^1(\Omega).
\end{align}

Taking $v = u - u_1 \in H_0^1$ in~\eqref{primal3A}
\begin{equation}\label{Du_est1A}
a(u - u_1, w) = \langle u - u_1, u - u_1 \rangle = \nr{u - u_1}^2_{L_2}.
\end{equation}

Combining~\eqref{primal1A} and~\eqref{primal2A} we have the Galerkin orthogonality result
\begin{equation}\label{Du_est2A}
a(u - u_1, v) = 0, \quad v \in \V_1.
\end{equation}

Then by~\eqref{Du_est1A} and~\eqref{Du_est2A} noting the interpolant of the dual solution $\cI^h w \in \V_1$
\begin{equation}\label{Du_est3Ap}
\nr{u - u_1}^2_{L_2} = a(u - u_1,w) = a(u - u_1, w - \cI^h w).
\end{equation}

Starting with~\eqref{Du_est3Ap} and applying continuity~\eqref{continuit1y}, interpolation estimate~\eqref{interpolation_estA} and elliptic regularity~\eqref{elliptic_regA}
\begin{align*}
\nr{u - u_1}_{L_2}^2 & \le  M_c \nr{u - u_1}_{H^1} \nr{w - \cI^h w}_{H^1} \\
& \le  M_c  \nr{u - u_1}_{H^1} C_\cI h_{\cT_1} |w|_{H^2}\\
& \le  K_R M_c C_\cI  h_{0} \nr{u - u_1}_{H^1}  \nr{u - u_1}_{L_2}.
\end{align*}

Canceling one factor of $\nr{u - u_1}_{L_2}$ and applying coercivity~\eqref{coerciv1e}
\begin{equation}\label{Du_est4A}
\nr{u - u_1}_{L_2} \le \f{M_c}{m_{\cE}} C_\cI  K_R h_0  \nrse{u - u_1}.
\end{equation}

Depending on the regularity of the boundary $\pa \Omega$ the solution $w$ may have less regularity: $w \in H^2_{\text{loc}(\Omega)}$ but $w \notin H^2(\Omega)$.  In particular, we may have $w \in H^{1+s}$ for some $s \in (0,1)$.  In that case obtain the more general estimate
\[
\nr{w - \cI^h w}_{H^1} \le \tilde C_\cI h_{0}^s|w|_{1+s}
\]
yielding
\[
\nr{u - u_1}_{L_2} \le \f{M_c}{m_{\cE}} \tilde C_\cI K_R h_0^s  \nrse{u - u_1}.
\]

The value of $s$ is found by considering all corners of boundary $\pa \Omega$.  Writing the interior angle at each corner by $\omega = \pi / \alpha$ it holds for  $\alpha > 0$ and arbitrary $\eps > 0$
\[
\omega = \pi / \alpha \implies w \in H^{1 + \alpha - \eps}
\]

and  if $\pi/(p_j+1) \le \omega \le \pi/p_j$ for a set of integers $p_j$ characterizing the corners of $\pa \Omega$
\[
\nr{w - \cI^h w}_{H^1} \le C h^s |w|_{1+s}
\]
 where $s = \min\{ p_j , 1\}$ and $s = 1$ in the case of a smooth boundary or a convex polyhedral domain.
Details may be found in~\cite{AOB01} and~\cite{SF73}.


\section*{Acknowledgments}
   \label{sec:ack}

MH was supported in part by NSF Awards 0715146 and 0915220
and  DOD/DTRA Award HDTRA-09-1-0036.
SP was supported in part by NSF Award~0715146.

\bibliographystyle{abbrv}
\bibliography{refs}
\end{document}